\documentclass[a4paper, reqno]{amsart}

\usepackage{Settings}

\title{Classification of objects in the singularity categories \\ of rational double points in arbitrary characteristics}

\author{Yuta Takashima}
\address{(Yuta TAKASHIMA) Department of Mathematical Sciences, Graduate School of Science,
Tokyo Metropolitan University, 1-1 Minamiosawa, Hachioji, Tokyo 192-0393, Japan.}
\email{yuta.takashima.m@gmail.com}



\begin{document}

\begin{abstract}
    We study rational double points over algebraically closed fields in arbitrary characteristics
    and completely classify the indecomposable objects in their singularity categories, which correspond to
    the vertices in their Auslander--Reiten quivers.
    Along the way, we present an alternative proof determining the configuration of these Auslander--Reiten quivers,
    and provide methods to handle the homotopy categories
    of matrix factorizations of isolated hypersurface singularities with computer algebra systems.
\end{abstract}

\maketitle

\setcounter{tocdepth}{2}
\tableofcontents

\section{Introduction}
Let $X$ be a finite-dimensional separated Noetherian scheme over an algebraically closed field $k$ such that
any coherent sheaf on $X$ is a quotient of a locally free sheaf of finite rank.
The \emph{singularity category} $\DD{sg} (X)$
is defined to be the Verdier quotient of the derived category $\DD{b} (\Coh X)$
by the full subcategory $\Perf (X)$ of perfect complexes.
Since $\DD{sg} (X)$ is trivial if and only if $X$ is smooth, 
$\DD{sg} (X)$ can be thought of what measures complexity of singularities.
Rather than $\DD{sg} (X)$ itself,
we are interested in the \emph{idempotent completion} $\overline{\DD{sg} (X)}$ (\cite{MR1813503}).
In fact, two schemes $X$ and $Y$
whose formal completions $\widehat{X}$ and $\widehat{Y}$ along their singular loci are isomorphic may have
non-equivalent singularity categories $\DD{sg} (X)$ and $\DD{sg} (Y)$,
whereas their idempotent completions $\overline{\DD{sg} (X)}$ and $\overline{\DD{sg} (Y)}$
are $k$-linear triangulated equivalent (\cite[Theorem~2.10]{MR2735755}).
If $X$ has only one isolated Gorenstein singular point $p \in X$, the idempotent completion $\overline{\DD{sg} (X)}$
turns out to be triangulated equivalent to $\DD{sg} (\widehat{\strsh}_{X, p})$
(\cite[Proposition~2.7]{MR2735755}, \cite[Proposition~A.1]{MR2776613} and \cite[Lemma~5.6]{MR2824483}).

In this paper,
we study the singularity category $\DD{sg} (\widehat{\strsh}_{X, p})$ of a rational double point $(X, p)$.
Let $S \coloneqq k \exdbra{x, y, z}$ and $f$ be the defining polynomial: $\widehat{\strsh}_{X, p} \cong S / \exgen{f}$.
As is well known ({\cite[Chapter~4]{MR4390795}} and \cite[Theorem~6.1]{MR570778}),
the singularity category $\DD{sg} (\widehat{\strsh}_{X, p})$ is $k$-linear triangulated equivalent
to the stable category $\sMCM(\widehat{\strsh}_{X, p})$ of maximal Cohen--Macaulay $\widehat{\strsh}_{X, p}$-modules
and the homotopy category $\HMF_S (f)$ of matrix factorizations.
Since these categories are $k$-linear Krull--Schmidt categories,
it is important to classify their indecomposable objects.
In the case that the characteristic of $k$ is zero,
Auslander has constituted a complete set of pairwise non-isomorphic 
indecomposable maximal Cohen--Macaulay $\widehat{\strsh}_{X, p}$-modules
and determined the Auslander--Reiten quiver in \cite{MR0816307}.
Furthermore, the corresponding indecomposable matrix factorizations are explicitly calculated
in \cite[Chapter~9, \S4]{MR2919145}.
Both arguments depend crucially on the fact
that rational double points in characteristic $0$ are quotient singularities.
In contrast to the characteristic $0$ case, rational double points in positive characteristics
are not quotient singularities in general (\cite[Section~3.9]{MR4996323}).
While it has been shown in \cite[Sections~2 and 3]{MR4788721} that Auslander's work mentioned above holds
for rational double points in positive characteristics
which are quotient singularities by linearly reductive finite group schemes (see \cite[Section~2.2]{MR4996323}),
the remaining cases have yet to be investigated.
This naturally leads us to ask how to find the indecomposable objects
in the singularity category $\DD{sg} (\widehat{\strsh}_{X, p})$ in arbitrary characteristic.
Our main result is the following theorem.

\begin{thm}[{\cref{thm:Main1,thm:Main2}}]
    \label{thm:Main}
    \begin{items}
        \item The matrix factorizations $\M_1, \M_2, \ldots, \M_n$ of type $A_n$ in \cref{subsec:A_n}
        constitute a complete set of pairwise non-isomorphic indecomposable objects in $\HMF_S (f)$ of type $A_n$,
        and the Auslander--Reiten quiver is given by
        
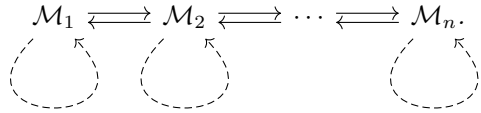
\[ 
\begin{tikzcd}[ampersand replacement=\&]
    \M_1 \ar[shift left=0.4ex]{r} \ar[out=225,in=315,loop, dashed] \&
        \M_2 \ar[shift left=0.4ex]{r} \ar[shift left=0.4ex]{l} \ar[out=225,in=315,loop, dashed] \&
            \cdots \ar[shift left=0.4ex]{r} \ar[shift left=0.4ex]{l} \&
                \M_n\rlap{.} \ar[shift left=0.4ex]{l} \ar[out=225,in=315,loop, dashed]
\end{tikzcd}
 \]
        \item The matrix factorizations $\M_1, \M_2, \ldots, \M_{2n}$ of type $D_{2n}^r$ in \cref{subsec:D_even^r}
        constitute a complete set of pairwise non-isomorphic indecomposable objects in $\HMF_S (f)$ of type $D_{2n}^r$,
        and the Auslander--Reiten quiver is given by
        
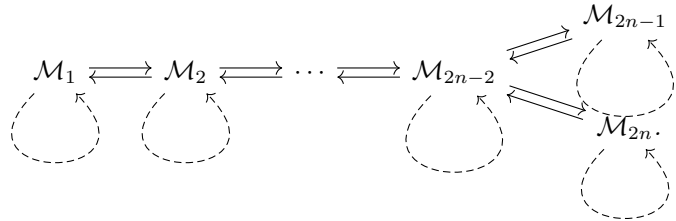
\[ 
\begin{tikzcd}[ampersand replacement=\&, row sep=tiny]
    \&
        \&
            \&
                \&
                    \M_{2n - 1} \ar[shift left=0.4ex]{dl} \ar[out=225,in=315,loop, dashed] \\
    \M_1 \ar[shift left=0.4ex]{r} \ar[out=225,in=315,loop, dashed] \&
        \M_2 \ar[shift left=0.4ex]{r} \ar[shift left=0.4ex]{l} \ar[out=225,in=315,loop, dashed] \&
            \cdots \ar[shift left=0.4ex]{r} \ar[shift left=0.4ex]{l} \&
                \M_{2n - 2} \ar[shift left=0.4ex]{l} \ar[shift left=0.4ex]{ur} \ar[shift left=0.4ex]{dr} \ar[out=225,in=315,loop, dashed] \&
                    \\
    \&
        \&
            \&
                \&
                    \M_{2n}\rlap{.} \ar[shift left=0.4ex]{ul} \ar[out=225,in=315,loop, dashed]
\end{tikzcd}
 \]
        \item The matrix factorizations $\M_1, \M_2, \ldots, \M_{2n + 1}$ of type $D_{2n + 1}^r$ in \cref{subsec:D_odd^r}
        constitute a complete set of pairwise non-isomorphic indecomposable objects in $\HMF_S (f)$ of type  $D_{2n + 1}^r$,
        and the Auslander--Reiten quiver is given by
        
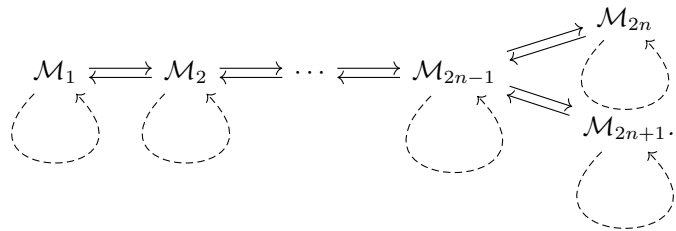
\[ 
\begin{tikzcd}[ampersand replacement=\&, row sep=tiny]
    \&
        \&
            \&
                \&
                    \M_{2n} \ar[shift left=0.4ex]{dl} \ar[out=225,in=315,loop, dashed] \\
    \M_1 \ar[shift left=0.4ex]{r} \ar[out=225,in=315,loop, dashed] \&
        \M_2 \ar[shift left=0.4ex]{r} \ar[shift left=0.4ex]{l} \ar[out=225,in=315,loop, dashed] \&
            \cdots \ar[shift left=0.4ex]{r} \ar[shift left=0.4ex]{l} \&
                \M_{2n - 1} \ar[shift left=0.4ex]{l} \ar[shift left=0.4ex]{ur} \ar[shift left=0.4ex]{dr} \ar[out=225,in=315,loop, dashed] \&
                    \\
    \&
        \&
            \&
                \&
                    \M_{2n + 1}\rlap{.} \ar[shift left=0.4ex]{ul} \ar[out=225,in=315,loop, dashed]
\end{tikzcd}
 \]
        \item The matrix factorizations $\M_1, \M_2, \ldots, \M_n$ of type $E_n^r$ in \crefrange{subsec:E_6^r}{subsec:E_8^rInChar2}
        constitute a complete set of pairwise non-isomorphic indecomposable objects in $\HMF_S (f)$ of type $E_n^r$,
        and the Auslander--Reiten quiver is given by
        
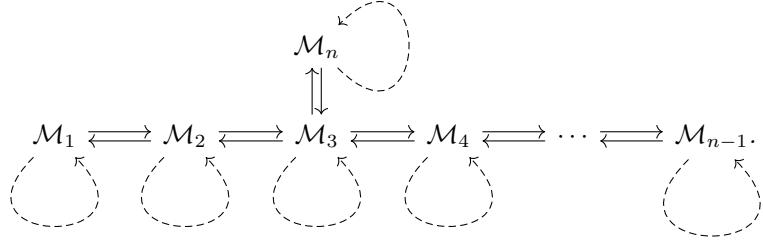
\[ 
\begin{tikzcd}[ampersand replacement=\&]
    \&
        \&
            \M_n \ar[shift left=0.4ex]{d} \ar[out=315,in=45,loop, dashed] \&
                \&
                    \&
                        \\
    \M_1 \ar[shift left=0.4ex]{r} \ar[out=225,in=315,loop, dashed] \&
        \M_2 \ar[shift left=0.4ex]{r} \ar[shift left=0.4ex]{l} \ar[out=225,in=315,loop, dashed] \&
            \M_3 \ar[shift left=0.4ex]{r} \ar[shift left=0.4ex]{l} \ar[shift left=0.4ex]{u} \ar[out=225,in=315,loop, dashed] \&
                \M_4 \ar[shift left=0.4ex]{r} \ar[shift left=0.4ex]{l} \ar[out=225,in=315,loop, dashed] \&
                    \cdots \ar[shift left=0.4ex]{r} \ar[shift left=0.4ex]{l} \&
                        \M_{n - 1}\rlap{.} \ar[shift left=0.4ex]{l} \ar[out=225,in=315,loop, dashed]
\end{tikzcd}
 \]
    \end{items}
\end{thm}

Here, types $D_n$ $(n \geq 4)$, $E_6$, $E_7$ and $E_8$
in Artin's classification (\cite[Section~3]{MR0450263}) of rational double points
are also denoted by $D_n^0$, $E_6^0$, $E_7^0$ and $E_8^0$ respectively.
\cref{thm:Main} is proved in \cref{sec:MainResults} with the aid of the computer algebra system \Singular{} (\cite{DGPS}).
In preparation for this, \cref{sec:MF} provides methods for computationally handling matrix factorizations
of isolated hypersurface singularities.
While the proof for types $E_n^r$ is exhausted by running a computer program
since $n$ and $r$ can only take finitely many values,
the other types require us to computationally find a candidate for
a complete set of pairwise non-isomorphic indecomposable objects in $\HMF_S (f)$,
and then rigorously prove its validity by applying the Auslander--Reiten theory of maximal Cohen--Macaulay modules,
which is summarized in \cref{sec:AR}.
Although it is known by \cite[Theorem~1]{MR0887498}
that the configuration of the Auslander--Reiten quiver of $\sMCM(\widehat{\strsh}_{X, p})$ is the double quiver
of the dual resolution graph of the rational double point $(X, p)$,
our argument in \cref{sec:AR} yields an alternative proof of this fact.

As an application of \cref{thm:Main}, using a computer algebra system,
we can determine whether the singularity category of a given rational double point is standard or not.
This topic will be addressed in detail in a forthcoming paper.

\begin{cvn}
    \begin{items}
        \item $k$ denotes an algebraically closed field.
        \item Any functor between $k$-linear categories is assumed to be $k$-linear.
        \item Let $S$ be a commutative ring, and $m$ and $n$ positive integers.
        $M_{m, n} (S)$ stands for the set of $m \times n$ matrices with entries in $S$.
        In particular when $m = n$, we simply write $M_n (S)$ for $M_{m, n} (S)$,
        and $I_n$ for the identity matrix. 
        Let $X \in M_{m, n} (S)$ be a matrix partitioned into row vectors as
        \[
            X =
            \begin{pmatrix}
                \bm{x}_1 \\
                \bm{x}_2 \\
                \vdots \\
                \bm{x}_m
            \end{pmatrix}.
        \]
        $\trans{X}$ is the transpose of $X$ and $\bm{v}(X)$ denotes
        $
            \trans{
                \begin{pmatrix}
                    \bm{x}_1 & \bm{x}_2 & \cdots & \bm{x}_n
                \end{pmatrix}
            }
        $. Also,
        \begin{align*}
            &A \otimes B \coloneqq 
            \begin{pmatrix}
                a_{1, 1}B & a_{1, 2}B & \cdots & a_{1, n}B \\
                a_{2, 1}B & a_{2, 2}B & \cdots & a_{2, n}B \\
                \vdots & \vdots & \ddots & \vdots \\
                a_{m, 1}B & a_{m, 2}B & \cdots & a_{m, n}B \\
            \end{pmatrix}
            & &\text{for $A = (a_{i, j})_{i, j} \in M_{m, n}(S)$ and $B \in M_{p, q}(S)$.}
        \end{align*}
        \item Let $\D$ be a $k$-linear category.
        The $k$-vector space of radical morphisms from one object $M$ to another $N$ is defined to be
        \begin{align*}
            \rad_{\D} (M, N)
            &\coloneqq \inset{\phi \in \Hom_{\D}(M, N)}
            {\text{$\id_M - \psi \circ \phi$ is isomorphic for any $\psi \in \Hom_{\D}(N, M)$}} \\
            &= \inset{\phi \in \Hom_{\D}(M, N)}
            {\text{$\id_N - \phi \circ \psi$ is isomorphic for any $\psi \in \Hom_{\D}(N, M)$}}.
        \end{align*}
        Also,
        \begin{align*}
            \rad^n_\D (M, N) &\coloneqq
            \ingen{\phi \in \Hom_\D (M, N)}{\phi \text{ is a composition of $n$ radical morphisms}}_k
                & &\text{for $n \geq 1$.}
        \end{align*}
        In particular when $\D$ is the category $\MCM (R)$ (resp. the stable category $\sMCM (R)$)
        of maximal Cohen--Macaulay modules over a commutative Noetherian local $k$-algebra $R$,
        we simply write $\rad_R$ (resp. $\srad_R$) for $\rad_\D$.
        \item Let $A$ be a ring.
        $\mod A$ stands for the category of finitely generated right $A$-modules.
        $M^\vee$ denotes the dual $\Hom_A (M, A)$ of a right $A$-module $M$.
        \item Let $X$ be a scheme.
        $\mathcal{M}^\vee$ denotes the dual $\shHom_X (\mathcal{M}, \strsh_X)$ of an $\strsh_X$-module $\mathcal{M}$.
    \end{items}
\end{cvn}

\begin{ack}
    The author would like to express his gratitude to Professor Hokuto Uehara for helpful discussions.
    Thanks are also due to the contributors of the computer algebra system \Singular{} (\cite{DGPS}),
    which was used for calculating matrix factorizations.
\end{ack}

\section{Auslander--Reiten theory}
\label{sec:AR}
Let $(X, p)$ be a quasi-projective $k$-variety $X$ with only one isolated Gorenstein singular point $p$,
and $R$ denote $\widehat{\strsh}_{X, p}$.
In this section, we discuss an approach to studying the singularity category $\DD{sg} (R)$
via the Auslander--Reiten theory of maximal Cohen--Macaulay $R$-modules
(see \cite{MR1079937}, \cite{MR2484725}, \cite{MR2919145} and \cite{MR2856194}).

\begin{thm}[{\cite[Chapter~4]{MR4390795}}]
    \label{thm:MCMToDsg}
    \begin{items}
        \item \label{thm:MCMToDsg:1} The category $\MCM (R)$ of maximal Cohen--Macaulay $R$-modules
        is a Frobenius subcategory of $\mod R$.
        \item \label{thm:MCMToDsg:2} The functor
        \[ \sMCM (R) \to \DD{sg} (R) \]
        induced by the composition
        $\MCM (R) \hookrightarrow \mod R \to \DD{sg} (R)$
        is a well-defined triangulated equivalence.
    \end{items}
\end{thm}

\begin{rem}
    Both $\MCM (R)$ and $\sMCM (R)$ are $k$-linear Krull--Schmidt categories
    since $R$ is a complete local $k$-algebra,
    and the latter is $\Hom$-finite by \cite{MR0842486}.
\end{rem}

\begin{thm}[{{\cite[Section~3]{MR0450263}} and \cite[Theorem~1.4]{MR1033443}}]
    \label{thm:SimpleSing}
    The following statements are equivalent\textup{:}
    \begin{items}
        \item \label{thm:SimpleSing:1} $(X, p)$ is a rational double point\textup{;}
        \item \label{thm:SimpleSing:2} $(X, p)$ is a two dimensional hypersurface singularity
        and, up to isomorphism, the number of indecomposable objects in $\MCM(R)$ is finite\textup{;}
        \item \label{thm:SimpleSing:3} $R$ is isomorphic to $k\exdbra{x, y, z} / \exgen{f}$
        with $f$ one of the polynomials listed in \cref{tab:RDPEq}.
        
\begin{longtable}{llll}
    \caption{Defining polynomials of rational double points.}
    \label{tab:RDPEq}
    \endfirsthead
    \multicolumn{4}{l}{\underline{Characteristic $0$ or not less than $7$}} \\ 
    & $A_n$ & $z^{n+1}+xy$ & for $n \ge 1$ \\
    & $D_{2n}$ & $z^2+x^2y+xy^n$ & for $n \ge 2$ \\
    & $D_{2n+1}$ & $z^2+x^2y+y^nz$ & for $n \ge 2$ \\
    & $E_6$ & $z^2+x^3+y^2z$ & \\
    & $E_7$ & $z^2+x^3+xy^3$ & \\
    & $E_8$ & $z^2+x^3+y^5$ & \\
    & & & \\
    \multicolumn{4}{l}{\underline{Characteristic $5$}} \\ 
    & $A_n$ & $z^{n+1}+xy$ & for $n \ge 1$ \\
    & $D_{2n}$ & $z^2+x^2y+xy^n$ & for $n \ge 2$ \\
    & $D_{2n+1}$ & $z^2+x^2y+y^nz$ & for $n \ge 2$ \\
    & $E_6$ & $z^2+x^3+y^2z$ & \\
    & $E_7$ & $z^2+x^3+xy^3$ & \\
    & $E^0_8$ & $z^2+x^3+y^5$ & \\
    & $E^1_8$ & $z^2+x^3+y^5+xy^4$ & \\
    & & & \\
    \multicolumn{4}{l}{\underline{Characteristic $3$}} \\ 
    & $A_n$ & $z^{n+1}+xy$ & for $n \ge 1$ \\
    & $D_{2n}$ & $z^2+x^2y+xy^n$ & for $n \ge 2$ \\
    & $D_{2n+1}$ & $z^2+x^2y+y^nz$ & for $n \ge 2$ \\
    & $E^0_6$ & $z^2+x^3+y^2z$ & \\
    & $E^1_6$ & $z^2+x^3+y^2z+xyz$ & \\
    & $E^0_7$ & $z^2+x^3+xy^3$ & \\
    & $E^1_7$ & $z^2+x^3+xy^3+x^2y^2$ & \\
    & $E^0_8$ & $z^2+x^3+y^5$ & \\
    & $E^1_8$ & $z^2+x^3+y^5+x^2y^3$ & \\
    & $E^2_8$ & $z^2+x^3+y^5+x^2y^2$ & \\
    & & & \\
    \multicolumn{4}{l}{\underline{Characteristic $2$}} \\ 
    & $A_n$ & $z^{n+1}+xy$ & for $n \ge 1$ \\
    & $D^0_{2n}$ & $z^2+x^2y+xy^n$ & for $n \ge 2$ \\
    & $D^r_{2n}$ & $z^2+x^2y+xy^n+xy^{n-r}z$ & for $n \ge 2$ and $1 \le r < n$ \\
    & $D^0_{2n+1}$ & $z^2+x^2y+y^nz$ & for $n \ge 2$ \\
    & $D^r_{2n+1} \quad$ & $z^2+x^2y+y^nz+xy^{n-r}z \quad$ & for $n \ge 2$ and $1 \le r < n$ \\
    & $E^0_6$ & $z^2+x^3+y^2z$ & \\
    & $E^1_6$ & $z^2+x^3+y^2z+xyz$ & \\
    & $E^0_7$ & $z^2+x^3+xy^3$ & \\
    & $E^1_7$ & $z^2+x^3+xy^3+x^2yz$ & \\
    & $E^2_7$ & $z^2+x^3+xy^3+y^3z$ & \\
    & $E^3_7$ & $z^2+x^3+xy^3+xyz$ & \\
    & $E^0_8$ & $z^2+x^3+y^5$ & \\
    & $E^1_8$ & $z^2+x^3+y^5+xy^3z$ & \\
    & $E^2_8$ & $z^2+x^3+y^5+xy^2z$ & \\
    & $E^3_8$ & $z^2+x^3+y^5+y^3z$ & \\
    & $E^4_8$ & $z^2+x^3+y^5+xyz$ &
\end{longtable}

    \end{items}
\end{thm}

\begin{rem}
    \begin{items}
        \item For uniform treatment independent of the characteristic of $k$,
        some of the polynomials in \cref{tab:RDPEq} differ from those in Artin's classification (\cite[Section~3]{MR0450263}).
        Simple changes of variables show that both lists are essentially the same.
        For type $E_6^1$ in characteristic $3$, for instance, the polynomial
        $z^2 + x^3 + y^2z + xyz$ in \cref{tab:RDPEq} transforms into $z^2 + x^3 + y^4 + x^2y^2$ in Artin's list
        as follows:
        \begin{align*}
            &z^2 + x^3 + y^2z + xyz \\
            &= z^2 + (-x^2 + y_1^2)z + x^3 & &\text{$(y_1 \coloneqq -x + y)$} \\
            &= z_1^2 + x^3(1 - x) - y_1^4 - x^2y_1^2 & &\text{$(z_1 \coloneqq z + x^2 - y_1^2)$} \\
            &= z_1^2 + \dfrac{x_1^3}{(1 + x_1)^4} - y_1^4 - \dfrac{x_1^2y_1^2}{(1 + x_1)^2}
                & &\text{$\expar{x_1 \coloneqq \dfrac{x}{1 - x}}$} \\
            &= z_1^2 + \dfrac{x_1^3}{(1 + x_1)^4} - \dfrac{y_2^4}{(1 + x_1)^4} - \dfrac{x_1^2y_2^2}{(1 + x_1)^4}
                & &\text{$(y_2 \coloneqq (1 + x_1)y_1)$} \\
            &= -\dfrac{z_2^2}{(1 + x_1)^4} + \dfrac{x_1^3}{(1 + x_1)^4} - \dfrac{y_2^4}{(1 + x_1)^4} - \dfrac{x_1^2y_2^2}{(1 + x_1)^4}
                & &\text{$(z_2 \coloneqq \sqrt{-1}(1 + x_1)^2z_1)$} \\
            &= -\dfrac{1}{(1 - x_2)^4}(z_2^2 + x_2^3 + y_2^4 + x_2^2y_2^2)
                & &\text{$(x_2 \coloneqq -x_1)$.}
        \end{align*}
        \item The type of a rational double point $(X, p)$ in \cref{tab:RDPEq} indicates
        what the dual resolution graph is.
        For instance, if the rational double point $(X, p)$ is of type $E^1_8$ in characteristic $2$, 
        its dual resolution graph is the Dynkin graph of type $E_8$.
    \end{items}
\end{rem}

The configuration of the Auslander--Reiten quiver of the singularity category of a rational double point
was first determined by Auslander and Reiten (\cite[Theorem~1]{MR0887498}).
Hereafter, assume $(X, p)$ be a rational double point and
we will present an alternative approach to obtaining the configuration of the Auslander--Reiten quiver.
Let $M_0, M_1, \ldots, M_n \in \MCM(R)$ constitute a complete set of pairwise non-isomorphic
indecomposable maximal Cohen--Macaulay $R$-modules with $M_0 = R$,
and set $M \coloneqq \bigoplus_{i = 0}^n M_i$, $A \coloneqq \End_R (M)$ and $J \coloneqq \rad_R(M, M)$.
Let $\pi \colon \widetilde{X} \to \Spec R$ be the minimal resolution.
We identify maximal Cohen--Macaulay $R$-modules with the associated coherent sheaves, and set
\begin{align*}
    &\widetilde{M} \coloneqq (\pi^*M)^{\vee \vee}
        & &\text{for $M \in \MCM(R)$.}
\end{align*}

\begin{prop}[{\cite{MR0769609}}]
    \label{prop:AV}
    Let $E_1, E_2, \ldots, E_n$ be the exceptional prime divisors.
    By relabeling the indices if necessary,
    \begin{align}
        &c_1(\widetilde{M_i}) \cdot E_j = \delta_{i, j}
            & &\text{for $i, j = 1, 2, \ldots, n$.} \label{eq:AV}
    \end{align}
\end{prop}

\begin{lem}[{e.g., \cite[Proposition~27.10]{MR1245487}}]
    \label{lem:PS}
    \begin{items}
        \item \label{lem:PS:1} For $i = 0, 1, \ldots, n$, let $e_i \colon M \to M$ be the composition of two canonical morphisms
        \[
            \begin{tikzcd}[ampersand replacement=\&]
                M \ar[two heads]{r} \&
                    M_i \ar[hook]{r} \&
                        M
            \end{tikzcd}
        \]
        and set $P_i \coloneqq e_i A \cong \Hom_R(M, M_i)$.
        Then $P_0, P_1, \ldots, P_n$ constitute a complete set of pairwise non-isomorphic
        indecomposable projective right $A$-modules.
        \item \label{lem:PS:2} For $i = 0, 1, \ldots, n$, set $S_i \coloneqq P_i / P_i J \cong \End_R(M_i) / \rad_R (M_i, M_i)$.
        Then $S_0, S_1, \ldots, S_n$ constitute a complete set of pairwise non-isomorphic
        simple right $A$-modules.
    \end{items}    
\end{lem}

We retain the notion of \cref{prop:AV} and \cref{lem:PS}, and assume \cref{eq:AV}.

\begin{prop}
    \label{prop:IrrToExt}
    \begin{align*}
        \rk_k \frac{\rad_R(M_i, M_j)}{\rad^2_R(M_i, M_j)} &= \rk_k \Ext_A^1(S_j, S_i)
            & &\text{for $i, j = 0, 1, \ldots, n$.}
    \end{align*}
\end{prop}

\begin{proof}
    We adopt an approach similar to that of \cite[III~Proposition~1.14]{MR1476671}.
    Take any $i, j = 0, 1, \ldots, n$. Then
    \[
        \frac{\rad_R(M_i, M_j)}{\rad^2_R(M_i, M_j)} = \frac{e_j J e_i}{e_j J^2 e_i}
        = \expar{\frac{e_j J}{e_j J^2}}e_i
        = \expar{\frac{P_j J}{P_j J^2}}e_i.
    \]
    Since $P_j J / P_j J^2$ is a right module over the semisimple ring $A / J$ and
    $S_0, S_1, \ldots, S_n$ also constitute a complete set of pairwise non-isomorphic simple right $A / J$-modules,
    there exist non-negative integers $m_0, m_1, \ldots, m_n$ such that
    \begin{align*}
        \frac{P_j J}{P_j J^2} &\cong \bigoplus_{l = 0}^n S_l^{\oplus m_l}
            & &\text{in $\mod A / J$.}
    \end{align*}
    Therefore as $k$-vector spaces,
    \begin{align*}
        \frac{\rad_R(M_i, M_j)}{\rad^2_R(M_i, M_j)} &\cong \bigoplus_{l = 0}^n (S_l e_i)^{\oplus m_l} \\
        &\cong \bigoplus_{l = 0}^n \Hom_A (P_i, S_l)^{\oplus m_l} \\
        &\cong \bigoplus_{l = 0}^n \Hom_A (S_i, S_l)^{\oplus m_l} \\
        &= \End_A(S_i)^{\oplus m_i}
            & &\text{by \cref{lem:PS}.}
    \end{align*}
    
    In what follows, we prove that $\End_A(S_i)^{\oplus m_i}$ is isomorphic to
    $\Ext_A^1(S_j, S_i)$ as right $A$-modules.
    Applying the functor $\Hom_A (-, S_i)$ to the short exact sequence
    \[
        \begin{tikzcd}[ampersand replacement=\&]
            0 \ar{r} \&
                P_j J \ar{r} \&
                    P_j \ar{r} \&
                        S_j \ar{r} \&
                            0
        \end{tikzcd}
    \]
    induces the long exact sequence
    \[
        \begin{tikzcd}[ampersand replacement=\&, row sep = tiny]
            0 \ar{r} \&
                \Hom_A(S_j, S_i) \ar{r} \&
                    \Hom_A(P_j, S_i) \ar{r} \&
                        \Hom_A(P_j J, S_i) \\
            \ar{r} \&
                \Ext^1_A(S_j, S_i) \ar{r} \&
                    \Ext^1_A(P_j, S_i) \ar{r} \&
                        \cdots\rlap{.}
        \end{tikzcd}
    \]
    Considering the morphism $\Hom_A(P_j, S_i) \to \Hom_A(P_j J, S_i)$ is zero
    and $\Ext^1_A(P_j, S_i) = 0$ by \cref{lem:PS}, we get $A$-linear isomorphisms
    \begin{align*}
        \Ext_A^1(S_j, S_i) &\cong \Hom_A(P_j J, S_i) \\
        &\cong \Hom_A \expar{\frac{P_j J}{P_j J^2}, S_i} \\
        &\cong \bigoplus_{l = 0}^n \Hom_A (S_l, S_i)^{\oplus m_l} \\
        &= \End_A (S_i)^{\oplus m_i} & &\text{by \cref{lem:PS}.} \qedhere
    \end{align*}
\end{proof}

\begin{lem}
    As rings,
    \[
        A \cong \End_{\widetilde{X}}(\widetilde{M}).
    \]
\end{lem}

\begin{proof}
    The long exact sequence induced by applying the functor $\Hom_{\widetilde{X}} (-, \widetilde{M})$
    to the short exact sequence
    \[
        \begin{tikzcd}[ampersand replacement=\&]
            0 \ar{r} \&
                \Tor_{\widetilde{X}} (\pi^* M) \ar{r} \&
                    \pi^* M \ar{r} \&
                        \widetilde{M} \ar{r} \&
                            0
        \end{tikzcd}
    \]
    tells us
    \[ \End_{\widetilde{X}} (\widetilde{M}) \cong \Hom_{\widetilde{X}} (\pi^* M, \widetilde{M}). \]
    Considering
    \begin{align*}
        \Hom_{\widetilde{X}} (\pi^* M, \widetilde{M}) \cong \Hom_R (M, \pi_* \widetilde{M}) \cong A,
    \end{align*}
    we get the ring isomorphism
    \begin{equation*}
        \begin{array}{ccc}
            \End_{\widetilde{X}}(\widetilde{M}) & \to & A \\
            \phi & \mapsto & \iota \circ (\pi_*\phi) \circ \iota^{-1}\rlap{,}
        \end{array}
    \end{equation*}
    where $\iota \colon \pi_* \widetilde{M} \to M$ is the canonical isomorphism.
\end{proof}

Via this ring isomorphism, $\End_{\widetilde{X}}(\widetilde{M})$ is identified with $A$.

\begin{prop}[{\cite[Section~3]{MR2057015}}]
    \label{prop:VdB}
    \begin{items}
        \item $\widetilde{M}$ is a tilting object of $\DD{b} (\Coh \widetilde{X})$\textup{:} the two functors
        \[
            \begin{tikzcd}[ampersand replacement=\&, column sep = huge]
                \DD{b} (\Coh \widetilde{X}) \ar[shift left=0.4ex]{r}{\RR \Hom_{\widetilde{X}}(\widetilde{M}, -)} \&
                    \DD{b} (\mod A) \ar[shift left=0.4ex]{l}{- \mathbin{\mathop{\otimes}\limits_{A}^{\mathbf{L}}} \widetilde{M}}
            \end{tikzcd}
        \]
        are mutually quasi-inverse triangulated equivalences. 
        \item In $\DD{b} (\Coh \widetilde{X})$,
        \begin{align*}
            S_i \dtens[A] \widetilde{M} &\cong \strsh_{E_i}(-1)[1]
                & &\text{for $i = 1, 2, \ldots, n$.}
        \end{align*}
    \end{items}
\end{prop}

\begin{thm}
    \label{thm:IrrConfig}
    \begin{align*}
        \rk_k\frac{\srad_R(M_i, M_j)}{\srad^2_R(M_i, M_j)} &= 2\delta_{i, j} + E_i \cdot E_j
            & &\text{for $i, j = 1, 2, \ldots, n$.}
    \end{align*}
\end{thm}

\begin{proof}
    Take any $i, j = 1, 2, \ldots, n$. As
    \begin{align*}
        \rk_k \frac{\srad_R(M_i, M_j)}{\srad^2_R(M_i, M_j)} &= \rk_k \frac{\rad_R(M_i, M_j)}{\rad^2_R(M_i, M_j)} \\
        &= \rk_k \Ext_A^1(S_j, S_i)
            & &\text{by \cref{prop:IrrToExt}} \\
        &= \rk_k \Ext_{\widetilde{X}}^1(\strsh_{E_j}(-1), \strsh_{E_i}(-1))
            & &\text{by \cref{prop:VdB},}
    \end{align*}
    it remains to show $\rk_k \Ext_{\widetilde{X}}^1(\strsh_{E_j}(-1), \strsh_{E_i}(-1))
    = 2\delta_{i, j} + E_i \cdot E_j$.
    Since $\rk_k \Hom_{\widetilde{X}} (\strsh_{E_j}(-1), \strsh_{E_i}(-1)) = \delta_{i, j}$
    by \cref{lem:PS} \cref*{lem:PS:2} and \cref{prop:VdB},
    \begin{align*}
        \chi (\strsh_{E_j}(-1), \strsh_{E_i}(-1))
        &\coloneqq \sum_{l = 0}^2 (-1)^l \rk_k \Ext_{\widetilde{X}}^l (\strsh_{E_j}(-1), \strsh_{E_i}(-1)) \\
        &= 2 \delta_{i, j} - \rk_k \Ext_{\widetilde{X}}^1 (\strsh_{E_j}(-1), \strsh_{E_i}(-1))
            & &\text{by Serre duality.}
    \end{align*}
    On the other hand, the Riemann--Roch theorem states
    \[
        \chi (\strsh_{E_j}(-1), \strsh_{E_i}(-1))
        = \int_{\widetilde{X}} \ch(\strsh_{E_j}(-1)^{\vee}) \ch(\strsh_{E_i}(-1)) \td_{\widetilde{X}}
        = - E_j \cdot E_i.
    \]
    Combining these equations, we obtain
    \[ \rk_k \Ext_{\widetilde{X}}^1(\strsh_{E_j}(-1), \strsh_{E_i}(-1)) = 2\delta_{i, j} + E_i \cdot E_j. \qedhere \]
\end{proof}

Since the Auslander--Reiten translation $\tau$ acts as the identity
on $\sMCM (R)$ (e.g., \cite[Theorem~2.47]{MR2484725}),
\cref{thm:IrrConfig} yields the following statement as a corollary.

\begin{cor}[{\cite[Theorem~1]{MR0887498}}]
    \label{cor:RDPQuiver}
    The configuration of the Auslander--Reiten quiver $(\varGamma, \tau)$ of $\sMCM (R)$ is the double quiver
    of the dual resolution graph of the rational double point $(X, p)$ as in \cref{tab:Config},
    where the Auslander--Reiten translation $\tau$ is depicted by dashed arrows.
    Here, types $D_n$ $(n \geq 4)$, $E_6$, $E_7$ and $E_8$ in \cref{tab:RDPEq} also be denoted by
    $D_n^0$, $E_6^0$, $E_7^0$ and $E_8^0$ respectively.
    
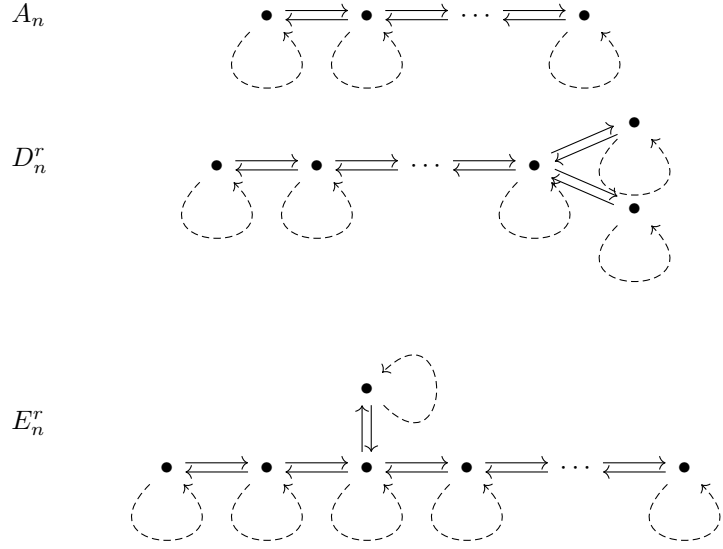
\begin{longtable}{lc}
    \caption{Configuration of the Auslander--Reiten quiver of $\sMCM(R)$.}
    \label{tab:Config}
    \endfirsthead
    $A_n$ & $
\begin{tikzcd}[ampersand replacement=\&]
    \bullet \ar[shift left=0.4ex]{r} \ar[out=225,in=315,loop, dashed] \&
        \bullet \ar[shift left=0.4ex]{r} \ar[shift left=0.4ex]{l} \ar[out=225,in=315,loop, dashed] \&
            \cdots \ar[shift left=0.4ex]{r} \ar[shift left=0.4ex]{l} \&
                \bullet \ar[shift left=0.4ex]{l} \ar[out=225,in=315,loop, dashed]
\end{tikzcd}
$ \\
    $D_n^r$ & $
\begin{tikzcd}[ampersand replacement=\&, row sep=tiny]
    \&
        \&
            \&
                \&
                    \bullet \ar[shift left=0.4ex]{dl} \ar[out=225,in=315,loop, dashed] \\
    \bullet \ar[shift left=0.4ex]{r} \ar[out=225,in=315,loop, dashed] \&
        \bullet \ar[shift left=0.4ex]{r} \ar[shift left=0.4ex]{l} \ar[out=225,in=315,loop, dashed] \&
            \cdots \ar[shift left=0.4ex]{r} \ar[shift left=0.4ex]{l} \&
                \bullet \ar[shift left=0.4ex]{l} \ar[shift left=0.4ex]{ur} \ar[shift left=0.4ex]{dr} \ar[out=225,in=315,loop, dashed] \&
                    \\
    \&
        \&
            \&
                \&
                    \bullet \ar[shift left=0.4ex]{ul} \ar[out=225,in=315,loop, dashed]
\end{tikzcd}
$ \\
    $E_n^r$ & $
\begin{tikzcd}[ampersand replacement=\&]
    \&
        \&
            \bullet \ar[shift left=0.4ex]{d} \ar[out=315,in=45,loop, dashed] \&
                \&
                    \&
                        \\
    \bullet \ar[shift left=0.4ex]{r} \ar[out=225,in=315,loop, dashed] \&
        \bullet \ar[shift left=0.4ex]{r} \ar[shift left=0.4ex]{l} \ar[out=225,in=315,loop, dashed] \&
            \bullet \ar[shift left=0.4ex]{r} \ar[shift left=0.4ex]{l} \ar[shift left=0.4ex]{u} \ar[out=225,in=315,loop, dashed] \&
                \bullet \ar[shift left=0.4ex]{r} \ar[shift left=0.4ex]{l} \ar[out=225,in=315,loop, dashed] \&
                    \cdots \ar[shift left=0.4ex]{r} \ar[shift left=0.4ex]{l} \&
                        \bullet \ar[shift left=0.4ex]{l} \ar[out=225,in=315,loop, dashed]
\end{tikzcd}
$
\end{longtable}

\end{cor}

\section{Matrix factorizations}
\label{sec:MF}
In preparation for the next section, we recall
that stable categories of maximal Cohen--Macaulay modules
are equivalent to homotopy categories of matrix factorizations as $k$-linear triangulated categories.
This description is more amenable to concrete computations.
Then we will discuss several basic calculations intended for a computer algebra system.

\begin{dfn}
    \label{def:MF}
    Let $S$ be a commutative ring and $f$ a regular element of $S$.
    \begin{items}
        \item \label{def:MF:1} A \emph{matrix factorization} of $f$ is a diagram
        \begin{equation*}
            \begin{tikzcd}[ampersand replacement=\&]
                S^n \ar{r}{B} \&
                    S^n \ar{r}{A} \&
                        S^n
            \end{tikzcd}
        \end{equation*}
        of free $S$-modules of finite rank, often referred to as $(A, B)$, such that 
        \[ A \circ B = B \circ A = f \cdot \id_{S^n}. \]
        A \emph{morphism} from one matrix factorization $(A, B)$ to another $(A', B')$
        is a pair $(X, Y)$ of $S$-linear maps such that the following diagram commutes:
        \begin{equation}   
            \label{eq:Mor}  
            \begin{tikzcd}[ampersand replacement=\&]
                S^m \ar{r}{B} \ar{d}[']{X} \&
                    S^m \ar{r}{A} \ar{d}{Y} \&
                        S^m \ar{d}{X} \\
                S^n \ar{r}[']{B'} \&
                    S^n \ar{r}[']{A'} \&
                        S^n\rlap{.}
            \end{tikzcd}
        \end{equation}
        The \emph{category of matrix factorizations}  of $f$ is denoted by $\MF_S (f)$.
        \item \label{def:MF:2} Two morphisms $(X, Y)$ and $(X', Y')$ from one matrix factorization $(A, B)$ to another $(A', B')$
        are \emph{homotopic}, if there exists a pair $(H_A, H_B)$ of $S$-linear maps as illustrated in the diagram
        \begin{equation*}     
            \begin{tikzcd}[ampersand replacement=\&]
                S^m \ar{r}{B} \ar{d}[']{X - X'} \&
                    S^m \ar{r}{A} \ar{d}{Y - Y'} \ar{dl}{H_B} \&
                        S^m \ar{d}{X - X'} \ar{dl}{H_A} \\
                S^n \ar{r}[']{B'} \&
                    S^n \ar{r}[']{A'} \&
                        S^n
            \end{tikzcd}
        \end{equation*}
        such that
        \begin{align}
            &X - X' = H_B \circ B + A' \circ H_A, \label{eq:Htp1} \\
            &Y - Y' = B' \circ H_B + H_A \circ A. \label{eq:Htp2}
        \end{align}
        Such a pair $(H_A, H_B)$ is called a \emph{homotopy} between the morphisms $(X, Y)$ and $(X', Y')$.
        In particular when $X' = 0$ and $Y' = 0$,
        the pair $(H_A, H_B)$ is called a \emph{null-homotopy} of the morphism $(X, Y)$.
        The \emph{homotopy category} $\HMF_S(f)$ of matrix factorizations of $f$
        is defined to be the category whose objects are the matrix factorizations of $f$,
        and whose morphisms are the homotopy classes of morphisms of matrix factorizations.
    \end{items}
\end{dfn}

\begin{rem}
    In \cref{def:MF}, the left rectangle of \cref{eq:Mor} commutes if and only if the right one does,
    and the conditions \cref{eq:Htp1} and \cref{eq:Htp2} are equivalent.
\end{rem}

\begin{thm}[{\cite[Chapter~4]{MR4390795} and {\cite[Theorem~6.1]{MR570778}}}]
    \label{thm:HMFToMCM}
    Let $S$ be a regular local ring and $f$ a non-zero element of $S$.
    Then the functor
    \[
        \begin{array}{ccc}
            \HMF_S (f) & \xrightarrow{\Coker} & \sMCM (S / \exgen{f}) \\
            (A, B) & \mapsto & \Coker A
        \end{array}
    \]
    is a well-defined triangulated equivalence.
\end{thm}

The following lemma will be used in the next section.

\begin{lem}
    \label{lem:SizeDown}
    Let $S$ be a regular local ring and $f$ a non-zero element of $S$, and set $R \coloneqq S / \exgen{f}$.

    \begin{items}
        \item \label{lem:SizeDown:1} Let $(A, B) \in \HMF_S (f)$ such that
        \begin{align*}
            A &= 
            \begin{pNiceArray}{c:c}[margin]
                * & * \\
                \hdottedline
                * & g I_n
            \end{pNiceArray}
            \in M_{2n}(S)
                & &\text{with $g \in S \setminus \exset{0}$.}
        \end{align*}
        Then
        \[ \Coker(A, B) \cong \exgen{\bm{b}_1, \bm{b}_2, \ldots, \bm{b}_{2n}}_{R} \subset R^n, \]
        where
        \begin{align*}
            B &= 
            \begin{pNiceArray}{cccc}[margin]
                \bm{b}_1 & \bm{b}_2 & \cdots & \bm{b}_{2n} \\
                \hdottedline
                \Block{1-4}{*}
            \end{pNiceArray}
            \in M_{2n}(S)
                & &\text{with $\bm{b}_1, \bm{b}_2, \ldots, \bm{b}_{2n}  \in S^n$.}
        \end{align*}
        \item \label{lem:SizeDown:2} Let $(A, B) \in \HMF_S (f)$ such that
        \begin{align*}
            A &= 
            \begin{pNiceArray}{c:c}[margin]
                g I_n & * \\
                \hdottedline
                * & *
            \end{pNiceArray}
            \in M_{2n}(S)
                & &\text{with $g \in S \setminus \exset{0}$.}
        \end{align*}
        Then
        \[ \Coker(A, B) \cong \exgen{\bm{b}_1, \bm{b}_2, \ldots, \bm{b}_{2n}}_{R} \subset R^n, \]
        where
        \begin{align*}
            B &= 
            \begin{pNiceArray}{cccc}[margin]
                \Block{1-4}{*} \\
                \hdottedline
                \bm{b}_1 & \bm{b}_2 & \cdots & \bm{b}_{2n}
            \end{pNiceArray}
            \in M_{2n}(S)
                & &\text{with $\bm{b}_1, \bm{b}_2, \ldots, \bm{b}_{2n}  \in S^n$.}
        \end{align*}
    \end{items}
\end{lem}

\begin{proof}
    We only consider \cref*{lem:SizeDown:1}. Since the sequence in $\mod R$
    \begin{align*}
        &\begin{tikzcd}[ampersand replacement=\&]
            \cdots \ar{r} \&
                R^{2n} \ar{r}{A} \&
                    R^{2n} \ar{r}{B} \&
                        R^{2n} \ar{r}{A} \&
                            R^{2n} \ar{r} \&
                                \Coker A \ar{r} \&
                                    0
        \end{tikzcd}
    \end{align*}
    is exact,
    \[ \Coker (A, B) = \Coker A = R^{2n} / \Im A = R^{2n} / \Ker B \cong \Im B. \]
    On the other hand, the composition
    \[
        \begin{tikzcd}[ampersand replacement=\&]
            \Im B = \Ker A \ar[hook]{r} \&
                R^{2n} \ar{r}{
                    \begin{pmatrix}
                        I_n & 0
                    \end{pmatrix}
                } \&
                    R^n\rlap{,}
        \end{tikzcd}
    \]
    denoted by $p \colon \Im B \to R^n$, is injective by the form of $A$.
    Therefore, we obtain $R$-linear isomorphisms
    \[
        \Coker(A, B) \cong \Im B \cong \Im p =
        \exgen{\bm{b}_1, \bm{b}_2, \ldots, \bm{b}_{2n}}_{R}.
        \qedhere
    \]
\end{proof}

Let $K$ be a finite extension over the prime field of $k$, and
set polynomial rings $S_K \coloneqq K[x_0, x_1, \ldots, x_d]$ and $S \coloneqq k[x_0, x_1, \ldots, x_d]$.
Consider $R_K \coloneqq S_K / \exgen{f}$ and $R \coloneqq S / \exgen{f}$ for some $f \in S_K \setminus \exset{0}$
such that $\Spec R$ is regular except at the origin.
By convention, the completions of these rings are always taken with respect to the origin, and
$\D_K$, $\D$ and $\widehat{\D}$ denote
$\HMF_{S_K} (f)$, $\HMF_{S} (f)$ and $\HMF_{\widehat{S}} (f)$ respectively.
For positive integers $m$ and $n$ and an extension ring $T$ of $S_K$, the matrix obtained from $X \in M_{m, n} (T)$
by substituting $0$ for $x_0, x_1, \ldots, x_d$ is denoted by $X_0$.

While we are interested in $\widehat{\D}$ and
intend to utilize a computer algebra system for calculating matrix factorizations,
computers cannot handle neither algebraically closed fields nor rings of formal power series
due to their infinite nature.
For the rest of this section, we will see that calculations in $\widehat{\D}$ can be reduced to those in $\D_K$,
thereby enabling the application of computer algebra systems.

\subsection{Homomorphisms}
\label{subsec:Homs}
Let $\M$ and $\M'$ be elements of $\D_K$ and set
    \begin{align*}
        &\M = (A, B) & &\text{with $A, B \in M_m(S_K)$,} \\
        &\M' = (A', B') & &\text{with $A', B' \in M_n(S_K)$.}
    \end{align*}
We give a method for finding a finite spanning set of the $k$-vector space $\Hom_{\widehat{\D}} (\M, \M')$
from the $K$-vector space $\Hom_{\D_K} (\M, \M')$.

\begin{prop}
    \label{prop:HomTheory}
    \begin{items}
        \item \label{prop:HomTheory:1}
        Let $\Sigma$ be a subring of $S_K$ and $F$ a flat $\Sigma$-algebra, and
        set $\widetilde{S} \coloneqq S_K \tens[\Sigma] F$ and $\widetilde{\D} \coloneqq \HMF_{\widetilde{S}} (f \tens 1)$.
        Then as $\widetilde{S} / \exgen{f \tens 1}$-modules
        \[
            \Hom_{\D_K} (\M, \M') \tens[\Sigma] F \cong \Hom_{\widetilde{\D}} (\M, \M').
        \]
        In particular when $\Sigma = K$ and $F = k$, the $R_K$-linear map
        \[
            \begin{array}{ccc}
                \Hom_{\D_K} (\M, \M') & \to & \Hom_{\D} (\M, \M') \\
                (X, Y) & \mapsto & (X, Y)
            \end{array}
        \]
        is a well-defined injection and
        \[ \rk_K \Hom_{\D_K} (\M, \M') = \rk_k \Hom_{\D} (\M, \M'). \]
        \item \label{prop:HomTheory:2} The $R$-linear map
        \[
            \begin{array}{ccc}
                \Hom_{\D} (\M, \M') & \to & \Hom_{\widehat{{\D}}} (\M, \M') \\
                (X, Y) & \mapsto & (X, Y)
            \end{array}
        \]
        is a well-defined isomorphism.
    \end{items}
\end{prop}

\begin{proof}
    \cref*{prop:HomTheory:1} Consider the $S_K$-linear maps
    \begin{equation}
        \begin{array}{ccccc}
            M_{n, m}(S_K) \times M_{n, m}(S_K) & \xrightarrow{\beta} & M_{n, m}(S_K) \times M_{n, m}(S_K) & \xrightarrow{\alpha} & M_{n, m}(S_K) \\
            (H_A, H_B) & \mapsto & (H_B \circ B + A' \circ H_A, B' \circ H_B + H_A \circ A) & & \\
            & & (X, Y) & \mapsto & X \circ A - A' \circ Y\rlap{.}
        \end{array}
        \label{eq:HomMatrix}
    \end{equation}
    Then we get $\widetilde{S} / \exgen{f \tens 1}$-linear isomorphisms
    \[
        \Hom_{\D_K} (\M, \M') \tens[\Sigma] F
        \cong \frac{\Ker \alpha}{\Im \beta} \tens[\Sigma] F
        \cong \frac{\Ker (\alpha \tens 1)}{\Im (\beta \tens 1)}
        \cong \Hom_{\widetilde{\D}} (\M, \M').
    \]

    \cref*{prop:HomTheory:2} Take $q = p / \exgen{f} \in \Spec R \setminus \exset{\exgen{x_0, x_1, \ldots, x_d}}$
    with $p \in \Spec S$, and set $\widetilde{S} \coloneqq S_p$ and
    $\widetilde{\D} \coloneqq \HMF_{\widetilde{S}} (f)$.
    Then
    \begin{align*}
        \Hom_{\D} (\M, \M')_q
        &\cong \Hom_{\D} (\M, \M') \tens[S] \widetilde{S} \\
        &\cong \Hom_{\widetilde{\D}} (\M, \M')
            & &\text{by \cref*{prop:HomTheory:1}.}
    \end{align*}
    Since $\widetilde{\D}$ is triangulated equivalent to $\sMCM(R_q)$ and $R_q$ is regular,
    \begin{align*}
        \Hom_{\D} (\M, \M')_q = 0.
    \end{align*}
    This implies the canonical $R$-linear map
    $\Hom_{\D} (\M, \M') \to \Hom_{\D} (\M, \M') \tens[R] \widehat{R}$
    is isomorphic. On the other hand as $\widehat{R}$-modules,
    \begin{align*}
        \Hom_{\D} (\M, \M') \tens[R] \widehat{R} &\cong \Hom_{\D} (\M, \M') \tens[S] \widehat{S} \\
        &\cong \Hom_{\widehat{\D}} (\M, \M')
            & &\text{by \cref*{prop:HomTheory:1}.}
    \end{align*}
    Combining these isomorphisms, we get the desired result.
\end{proof}

This proposition shows the $k$-vector space $\Hom_{\widehat{\D}} (\M, \M')$ turns out to be generated
by a spanning set of the $K$-vector space $\Hom_{\D_K} (\M, \M')$.
Let $\C_K$ denote $\MF_{S_K} (f)$ and consider the $S_K$-linear isomorphism
\[
    \begin{array}{ccc}
        \Hom_{\C_K} (\M, \M')
            & \to
                & \Ker \expar{\begin{pNiceArray}{c:c}[margin]
                    I_n \tens \trans{A} & - A' \tens I_m
                \end{pNiceArray}
                \colon S_K^{2mn} \to S_K^{mn}}\\
        (X, Y)
            & \mapsto
                & \begin{pmatrix}
                    \bm{v}(X)  \\
                    \bm{v}(Y)
                \end{pmatrix}.
    \end{array}
\]
Since a generating set  of
the target can be found using a computer algebra system,
we can construct a generating set $(X_1, Y_1), (X_2, Y_2), \ldots, (X_l, Y_l)$
of the $S_K$-module $\Hom_{\C_K} (\M, \M')$. Set
\begin{align*}
    &C \coloneqq \begin{pNiceArray}{c:c}[margin]
        I_n \tens \trans{A} & B' \tens I_m
    \end{pNiceArray}
        & &\text{(cf. \cref{subsec:Zeros}),} \\
    &N_i \coloneqq \Ker \expar{S_K \xrightarrow{\bm{v}(Y_i)} S_K^{mn} \twoheadrightarrow S_K^{mn} / \Im C}
        & &\text{for $i = 1, 2, \cdots, l$.}
\end{align*}
Then for any $i = 1, 2, \cdots, l$, the $K$-linear map
\[
    \begin{array}{ccc}
        S_K / N_i & \to & \Hom_{\D_K} (\M, \M') \\
        g  & \mapsto & g \cdot (X_i, Y_i)
    \end{array}
\]
is a well-defined injection and,
considering $\D_K$ is $\Hom$-finite, we can also find a generating set 
$g^{(1)}_i, g^{(2)}_i, \ldots, \allowbreak g^{(l_i)}_i$ of the source $K$-vector space using a computer algebra system.
Consequently,
\[ \Hom_{\widehat{\D}} (\M, \M') = \ingen{g^{(j)}_i \cdot (X_i, Y_i)}{1 \leq i \leq l \wedge 1 \leq j \leq l_i}_k. \]
Note that this $k$-vector space is spanned by morphisms in $\D_K$.

\subsection{Linear relations among homomorphisms}
Let $\M$ and $\M'$ be elements of $\D_K$ and set
    \begin{align*}
        &\M = (A, B) & &\text{with $A, B \in M_m(S_K)$,} \\
        &\M' = (A', B') & &\text{with $A', B' \in M_n(S_K)$.}
    \end{align*}
We give a method for finding the linear relations among the morphisms
$(X_1, Y_1), (X_2, Y_2), \ldots, (X_l, Y_l) \in \Hom_{\D_K} (\M, \M')$.
Consider $S_K$-linear maps \cref{eq:HomMatrix} again and set the $S_K$-linear isomorphism
\[
    \begin{array}{ccc}
        M_{n, m}(S_K) \times M_{n, m}(S_K) & \xrightarrow{\bm{v}} & S_K^{2mn} \\
        (X, Y)
            & \mapsto 
            & \begin{pmatrix}
                \bm{v}(X) \\
                \bm{v}(Y)
            \end{pmatrix}\rlap{.}
    \end{array}
\]
Then in $\mod S_K$,
\[
    \Hom_{\D_K} (\M, \M')
    \cong \dfrac{\Ker \alpha}{\Im \beta}
    \cong \dfrac{\bm{v}(\Ker \alpha)}{\bm{v}(\Im \beta)}
    \hookrightarrow \dfrac{S_K^{2mn}}{\bm{v}(\Im \beta)}.
\]
On the other hand, fixing a global module ordering on the $S_K$-module $S_K^{2mn}$,
we get the well-defined $K$-linear injection
\[
    \begin{array}{ccc}
        S_K^{2mn} / \bm{v}(\Im \beta)
            & \to
                & S_K^{2mn}\\
        \bm{w}
            & \mapsto
                & \NF(\bm{w})\rlap{,}
    \end{array}
\]
where $\NF \colon S_K^{2mn} \to S_K^{2mn}$ is the normal form
with respect to a Gr\"{o}bner basis of $\bm{v}(\Im \beta)$ (e.g., \cite[Excercise~2.3.3]{MR2363237}).
Therefore, finding the linear relations among the morphisms
$(X_1, Y_1), (X_2, Y_2), \ldots, (X_l, Y_l) \in \Hom_{\D_K} (\M, \M')$ reduces to calculating
the linear relations among the vectors
$\NF(\bm{v}(X_1, Y_1)), \NF(\bm{v}(X_2, Y_2)), \allowbreak \ldots, \NF(\bm{v}(X_l, Y_l)) \in S_K^{2mn}$,
which can be done by using a computer algebra system.

\subsection{Zero morphisms}
\label{subsec:Zeros}
We give a condition for a morphism in $\D_K$ to be zero in $\widehat{\D}$.
Let $(X, Y) \colon \M \to \M'$ be a morphism in $\D_K$ and set
\begin{align*}
    &\M = (A, B) & &\text{with $A, B \in M_m(S_K)$,} \\
    &\M' = (A', B') & &\text{with $A', B' \in M_n(S_K)$.}
\end{align*}
By \cref{prop:HomTheory}, the morphism $(X, Y)$ is zero in $\widehat{\D}$ if and only if so is it in $\D_K$.
The map
\[
    \begin{array}{ccc}
        \inset{(H_A, H_B) \in M_{m, n}(S_K)^2}{
            \begin{array}{ll}
                \text{$(H_A, H_B)$ is a null-} \\
                \text{homotopy of $(X, Y)$}
            \end{array}
        }
            & \to
                & \inset{\bm{h} \in S_K^{2mn}}{
                    \begin{pNiceArray}{c:c}[margin]
                        I_n \tens \trans{A} & B' \tens I_m
                    \end{pNiceArray}
                    \bm{h}
                    = \bm{v}(Y)
                }\\
        (H_A, H_B)
            & \mapsto
                & \begin{pmatrix}
                    \bm{v}(H_A)  \\
                    \bm{v}(H_B)
                \end{pmatrix}
    \end{array}
\]
being bijective, the morphism $(X, Y)$ is zero in $\widehat{\D}$ precisely when
\[
    \inset{\bm{h} \in S_K^{2mn}}{
        \begin{pNiceArray}{c:c}[margin]
            I_n \tens \trans{A} & B' \tens I_m
        \end{pNiceArray}
        \bm{h}
        = \bm{v}(Y)
    }
    \neq \emptyset.
\]
Whether this condition holds or not can be verified by computer calculation.

\subsection{Isomorphisms}
We give a condition for two matrix factorizations in $\D_K$
of the same size and with no constant terms in their entries
to be isomorphic in $\widehat{\D}$.
Note that $X \in M_n (\widehat{S})$ is invertible if and only if so is $X_0 \in M_n (k)$.
Then the following lemma is straightforward.

\begin{lem}
    \label{lem:IsomTheory}
    Let $(X, Y) \colon \M \to \N$ be a morphism in $\widehat{\D}$
    between matrix factorizations of the same size and with no constant terms in their entries.
    Then the following statements are equivalent\textup{:}
    \begin{items}
        \item $(X, Y)$ is an isomorphism\textup{;}
        \item $X_0$ and $Y_0$ are invertible.
    \end{items}
\end{lem}

Furthermore, a simple observation yields the next lemma.

\begin{lem}
    \label{lem:Solution}
    Let $g_1, g_2, \ldots, g_l \in S_K$ and set an ideal $I \coloneqq \exgen{g_1, g_2, \ldots, g_l}$.
    Then the following statements are equivalent\textup{:}
    \begin{items}
        \item The system of equations
        \[
            \left\{
            \begin{aligned}
                g_1(x_0, x_1, \ldots, x_d) &= 0 \\
                g_2(x_0, x_1, \ldots, x_d) &= 0 \\
                &\vdots \\
                g_l(x_0, x_1, \ldots, x_d) &= 0
            \end{aligned}
            \right.
        \]
        admits a solution in $k^{d + 1}$\textup{;}
        \item $I \neq S_K$.
    \end{items}
\end{lem}

Let $\M$ and $\N$ be matrix factorizations in $\D_K$ of the same size and with no constant terms in their entries, and set
\[ \Hom_{\D_K} (\M, \N) = \exgen{\phi_1, \phi_2, \ldots, \phi_l}_K. \]
Consider the polynomial ring $K[\lambda_1, \lambda_2, \ldots, \lambda_l, \mu_X, \mu_Y]$ and let
\[ (X, Y) \coloneqq \lambda_1 \cdot \phi_1 + \lambda_2 \cdot \phi_2 + \cdots + \lambda_l \cdot \phi_l. \]
Then by \cref{lem:IsomTheory} and \cref{lem:Solution}, $\M$ and $\N$ are isomorphic in $\widehat{\D}$ precisely when
\[
    \exgen{\mu_X \det(X_0) - 1, \mu_Y \det(Y_0) - 1}
    \neq K[\lambda_1, \lambda_2, \ldots, \lambda_l, \mu_X, \mu_Y].
\]
Whether this condition holds or not can be verified by computer calculation.

\subsection{Radical morphisms}
Let $\M$ and $\N$ be matrix factorizations in $\D_K$ 
which are indecomposable in $\widehat{\D}$ and have no constant terms in their entries.
We give a method for finding a finite spanning set of the $k$-vector space
$\rad_{\widehat{\D}}^n (\M, \N)$ for $n \geq 1$ from the $K$-vector space $\rad_{\D_K}^n (\M, \N)$.
Since the ideals $\rad_{\widehat{\D}}^n$ for $n \geq 2$ can be computed recursively from $\rad_{\widehat{\D}}$
and $\rad_{\widehat{\D}} (\M, \N) = \Hom_{\widehat{\D}} (\M, \N)$ when $\M$ and $\N$ are not isomorphic,
it suffices to consider $\rad_{\widehat{\D}} (\M, \M)$.
Note that since $k$ is an algebraically closed field, as $k$-vector spaces
\begin{equation}
    \End_{\widehat{\D}} (\M) \cong k \cdot \id_{\M} \oplus \rad_{\widehat{\D}} (\M, \M).
    \label{eq:EndDecomp}
\end{equation}

\begin{lem}
\label{lem:MorDecomp}
Let $(X, Y) \colon \M \to \M$ be a morphism in $\D_K$ and
\[ (X, Y) = \lambda \cdot \id_{\M} + (X', Y') \]
be a decomposition along \cref{eq:EndDecomp}.
Then $\lambda$ belongs to $K$ and is the unique eigenvalue of $X_0$ and $Y_0$.
\end{lem}

\begin{proof}
    Let $m \in \ZZ_{\geq 1}$ be the size of $\M$. Since
    \begin{align}
        &X_0 = \lambda \cdot I_m + X'_0, \label{eq:Jordan} \\
        &Y_0 = \lambda \cdot I_m + Y'_0 \notag
    \end{align}
    are Jordan decompositions, $\lambda \in k$ is the unique eigenvalue of $X_0$ and $Y_0$.
    In order to show that $\lambda$ belongs to $K$, we divide into two cases according to the characteristic $p$ of $k$.
    If $p = 0$, taking the trace of both sides of \cref{eq:Jordan}, we obtain
    \[ \lambda = \dfrac{\tr(X_0)}{m} \in K. \]
    Otherwise, there exists a positive integer $a$ such that
    \[ \lambda^{p^a} = \lambda. \]
    Also, since $X'_0 \in M_m(k)$ is nilpotent, there exists a positive integer $b$ such that
    \begin{align*}
        &(X'_0)^{q} = 0 & &\text{where $q \coloneqq p^{ab}$.}
    \end{align*}
    Therefore,
    \begin{align*}
        \lambda \cdot I_m
        = \lambda^{q} \cdot I_m
        = \lambda^{q} \cdot I_m + (X'_0)^{q}
        = X_0^{q}
        \in M_m(K).
    \end{align*}
    This implies $\lambda \in K$.
\end{proof}

Set
\begin{align*}
    \End_{\D_K} (\M) = \exgen{(X_1, Y_1), (X_2, Y_2), \ldots, (X_l, Y_l)}_K,
\end{align*}
and let $\lambda_i \in K$ be the unique eigenvalue of $(X_i)_0$ and $(Y_i)_0$ for $i = 1, 2, \ldots, l$.
Then by \cref{lem:MorDecomp},
\[
    \rad_{\widehat{\D}} (\M, \M) =\exgen{
        (X_1 - \lambda_1, Y_1 - \lambda_1),
        (X_2 - \lambda_2, Y_2 - \lambda_2),
        \ldots,
        (X_l - \lambda_l, Y_l - \lambda_l)
    }_k.
\]
Note that this $k$-vector space is spanned by morphisms in $\D_K$.

\subsection{Auslander--Reiten triangles}
Let $\M$ be a matrix factorization in $\D_K$ which is indecomposable in $\widehat{\D}$
and has no constant terms in its entries.
We give a method for constructing the Auslander--Reiten triangle in $\widehat{\D}$
\begin{align*}
    \begin{tikzcd}[ampersand replacement=\&]
        \tau(\M) \ar{r}{} \&
            \mathcal{L} \ar{r}{} \&
                \M \ar{r}{} \&
                    \tau(\M)[1]\rlap{,}
    \end{tikzcd}
\end{align*}
where $\tau = [d - 2]$ is the Auslander--Reiten translation (e.g., \cite[Theorem~2.47]{MR2484725}).
Let the socle of the right $\End_{\widehat{\D}} (\M)$-module $\Hom_{\widehat{\D}} (\M, \tau(\M)[1])$
be denoted by $\Soc(\Hom_{\widehat{\D}} (\M, \tau(\M)[1]))$.
As is well known (e.g., \cite[Theorem~13.8]{MR2919145}), $\rk_k \Soc(\Hom_{\widehat{\D}} (\M, \tau(\M)[1])) = 1$
and the following proposition holds.

\begin{prop}
    Let
    \begin{align}
        &\begin{tikzcd}[ampersand replacement=\&]
            \tau(\M) \ar{r}{} \&
                \mathcal{L} \ar{r}{} \&
                    \M \ar{r}{\psi} \&
                        \tau(\M)[1]
        \end{tikzcd}
        \label{eq:TargetTri}
    \end{align}
    be an exact triangle in $\widehat{\D}$.
    Then the following statements are equivalent\textup{:}
    \begin{items}
        \item The exact triangle \cref{eq:TargetTri} is an Auslander--Reiten triangle\textup{;}
        \item The morphism $\psi$ is a non-zero element of $\Soc(\Hom_{\widehat{\D}} (\M, \tau(\M)[1]))$.
    \end{items}
\end{prop}

Therefore, it suffices to find a non-zero morphism $\psi \colon \M \to \tau(\M)[1]$ in $\widehat{\D}$
such that $\psi \circ \rad_{\widehat{\D}}(\M, \M) = 0$.
Let $\phi_1, \phi_2, \ldots, \phi_m$ (resp. $\psi_1, \psi_2, \ldots, \psi_n$) constitute a basis of
$K$-vector space $\rad_{\D_K} (\M, \M)$ (resp. $\Hom_{\D_K} (\M, \tau(\M)[1])$).
For each $i = 1, 2, \ldots, m$, using a computer algebra system, we can find the unique matrix
$C_i \in M_n(K)$ such that
\[
    \begin{pmatrix}
        \psi_1 \circ \phi_i & \psi_2 \circ \phi_i & \cdots & \psi_n \circ \phi_i
    \end{pmatrix}
    = \begin{pmatrix}
        \psi_1 & \psi_2 & \cdots & \psi_n
    \end{pmatrix}
    C_i.
\]
Set
\[
    C \coloneqq
    \begin{pmatrix}
        C_1 \\
        C_2 \\
        \vdots \\
        C_n
    \end{pmatrix}
    \in M_{mn, n}(K).
\]
Then we get the $k$-linear isomorphism
\[
    \begin{array}{ccc}
        \Ker \expar{C \colon k^n \to k^{mn}}
            & \to
                & \Soc(\Hom_{\widehat{\D}} (\M, \tau(\M)[1])) \\
        \trans{
            \begin{pmatrix}
                \lambda_1 & \lambda_2 & \cdots & \lambda_n
            \end{pmatrix}
        }
            & \mapsto
                & \sum\limits_{j = 1}^n \lambda_j \cdot \psi_j\rlap{.}
    \end{array}
\]
Since $\Soc(\Hom_{\widehat{\D}} (\M, \tau(\M)[1])) \neq 0$, it follows that
$\Ker \expar{C \colon k^n \to k^{mn}} \neq 0$ (i.e., $\rk C < n$).
This implies $\Ker \expar{C \colon K^n \to K^{mn}} \neq 0$.
Consequently, by taking a non-zero element
$\trans{\begin{pmatrix}
    \lambda_1 & \lambda_2 & \cdots & \lambda_n
\end{pmatrix}}$
of this $K$-vector space using a computer algebra system,
we find that $\psi \coloneqq \sum_{j = 1}^n \lambda_j \cdot \psi_j$ is
a non-zero element of $\Soc (\Hom_{\widehat{\D}} (\M, \tau(\M)[1]))$,
which induces the desired Auslander--Reiten triangle.

\subsection{Indecomposable decompositions}
\label{subsec:IndecompDecomps}
Let $\M$ and $\N$ be matrix factorizations in $\D_K$ such that
$\M$ is indecomposable and a direct summand of $\N$ in $\widehat{\D}$,
and assume $\M$ has no constant terms in its entries.
We give a method for constructing a complement of $\M$ which belongs to $\D_K$.
Let $\iota \colon \M \to \N$ and $\pi \colon \N \to \M$ denote the canonical morphisms in $\widehat{\D}$ and set
\begin{align*}
    &\iota = \sum_{i = 1}^m \lambda_i \cdot \phi_i
        & &\text{with $\lambda_1, \lambda_2, \ldots, \lambda_m \in k$}, \\
    &\pi = \sum_{j = 1}^n \mu_j \cdot \psi_j
        & &\text{with $\mu_1, \mu_2, \ldots, \mu_n \in k$},
\end{align*}
where
\begin{align*}
    \Hom_{\D_K} (\M, \N) = \exgen{\phi_1, \phi_2, \ldots, \phi_m}_K
    \text{ and }
    \Hom_{\D_K} (\N, \M) = \exgen{\psi_1, \psi_2, \ldots, \psi_n}_K.
\end{align*}
Then
\[ \sum_{i, j} (\lambda_i \mu_j) \cdot (\psi_j \circ \phi_i) = \id_{\M}.  \]
Since $\End_{\widehat{\D}}(\M)$ is a local ring, there exist some $i = 1, 2, \ldots, m$ and
$j = 1, 2, \ldots, n$ such that $\psi_j \circ \phi_i$ is an isomorphism.
We can find such $i$ and $j$ by using a computer algebra system and obtain
\[ \expar{(\psi_j \circ \phi_i)^{-1} \circ \psi_j} \circ \phi_i = \id_{\M}. \]
This implies that $\phi_i$ is a splitting monomorphism. Therefore, considering the exact triangle in $\widehat{\D}$
\begin{align*}
    \begin{tikzcd}[ampersand replacement=\&]
        \M \ar{r}{\phi_i} \&
            \N \ar{r}{} \&
                \Cone(\phi_i) \ar{r}{} \&
                    \M[1]
    \end{tikzcd}
\end{align*}
leads us to get a complement $\Cone(\phi_i)$ of $\M$, which belongs to $\D_K$.

\section{Main results}
\label{sec:MainResults}
Let $(X, p)$ be a rational double point and $R$ denote $\widehat{\strsh}_{X, p}$.
In this section, we determine the indecomposable objects corresponding to the vertices in the Auslander--Reiten quiver
of the singularity category $\DD{sg} (R)$, which is triangulated equivalent to $\sMCM(R)$ and $\HMF_S (f)$,
where $S \coloneqq k \exdbra{x, y, z}$ and $f$ is one of the polynomials listed in \cref{tab:RDPEq}
such that $R \cong S / \exgen{f}$.
For convenience, let types $D_n$ $(n \geq 4)$, $E_6$, $E_7$ and $E_8$ in \cref{tab:RDPEq} also be denoted by
$D_n^0$, $E_6^0$, $E_7^0$ and $E_8^0$ respectively.

For types $E_n^r$, since $n$ and $r$ can only take finitely many values,
we obtain our first main theorem by applying the techniques described in \crefrange{subsec:Homs}{subsec:IndecompDecomps}
with the aid of the computer algebra system \Singular{} (\cite{DGPS}).

\begin{thm}
    \label{thm:Main1}
    The matrix factorizations $\M_1, \M_2, \ldots, \M_n$ of type $E_n^r$ in \crefrange{subsec:E_6^r}{subsec:E_8^rInChar2}
    constitute a complete set of pairwise non-isomorphic indecomposable objects in $\HMF_S (f)$ of type $E_n^r$,
    and the Auslander--Reiten quiver is given by
    
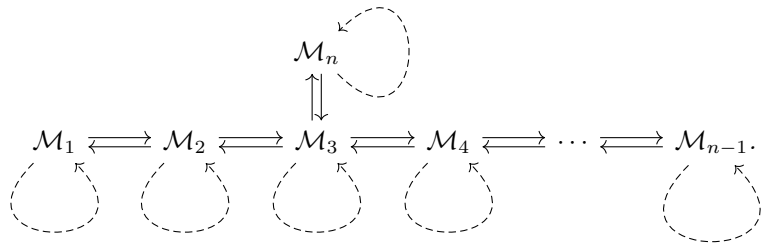
\[ 
\begin{tikzcd}[ampersand replacement=\&]
    \&
        \&
            \M_n \ar[shift left=0.4ex]{d} \ar[out=315,in=45,loop, dashed] \&
                \&
                    \&
                        \\
    \M_1 \ar[shift left=0.4ex]{r} \ar[out=225,in=315,loop, dashed] \&
        \M_2 \ar[shift left=0.4ex]{r} \ar[shift left=0.4ex]{l} \ar[out=225,in=315,loop, dashed] \&
            \M_3 \ar[shift left=0.4ex]{r} \ar[shift left=0.4ex]{l} \ar[shift left=0.4ex]{u} \ar[out=225,in=315,loop, dashed] \&
                \M_4 \ar[shift left=0.4ex]{r} \ar[shift left=0.4ex]{l} \ar[out=225,in=315,loop, dashed] \&
                    \cdots \ar[shift left=0.4ex]{r} \ar[shift left=0.4ex]{l} \&
                        \M_{n - 1}\rlap{.} \ar[shift left=0.4ex]{l} \ar[out=225,in=315,loop, dashed]
\end{tikzcd}
 \]
\end{thm}

\begin{proof}
    Since a rational double point of type $E_n$ in characteristic at least $7$
    is a quotient singularity by a finite subgroup of $\SL_2(k)$ just as in characteristic $0$,
    it suffices to consider the case that the characteristic of $k$ is less than $7$.

    \textit{Step 1.} We show that $\M_1$ and $\M_{n - 1}$ are indecomposable.
    Let $M_1$ and $M_{n - 1}$ be maximal Cohen--Macaulay $R$-modules
    induced from $\M_1$ and $\M_{n - 1}$ by \cref{lem:SizeDown} respectively. Then
    \begin{align*}
        &\rk_R M_{1} =
        \begin{cases}
            1 & \text{if $n = 6$,} \\
            2 & \text{if $n = 7, 8$,}
        \end{cases} \\
        &\rk_R M_{n - 1} =
        \begin{cases}
            1 & \text{if $n = 6, 7$,} \\
            2 & \text{if $n = 8$.}
        \end{cases}
    \end{align*}
    Considering that the ranks of indecomposable maximal Cohen--Macaulay $R$-modules corresponding to
    the vertices of the Auslander--Reiten quiver of $\sMCM(R)$ coincide with
    the coefficients appearing in the fundamental cycle of the rational double point $(X, p)$ (\cite{MR0769609}) and
    \begin{align*}
        &\M_1 \ncong \M_{n - 1} \oplus \M_{n - 1} & &\text{if $n = 7$,}
    \end{align*}
    we see that $\M_1$ and $\M_{n - 1}$ are indecomposable.

    \textit{Step 2.} We produce $\M_{n - 2}, \M_{n - 3}, \ldots, \M_{3}$ (resp. $\M_2$ and $\M_3$)
    recursively from $\M_{n - 1}$ (resp. $\M_{1}$).
    We first obtain the indecomposable matrix factorization $\M_{n - 2}$ from the Auslander--Reiten triangle
    \begin{align*}
        \begin{tikzcd}[ampersand replacement=\&]
            \M_{n - 1} \ar{r}{} \&
                \mathcal{L}_{n - 1} \ar{r}{} \&
                    \M_{n - 1} \ar{r} \&
                        \M_{n - 1}[1]
        \end{tikzcd}
    \end{align*}    
    by splitting off the trivial matrix factorizations $(1, f)$ and $(f, 1)$ from $\mathcal{L}_{n - 1}$.
    Next, the Auslander--Reiten triangle
    \begin{align*}
        \begin{tikzcd}[ampersand replacement=\&]
            \M_{n - 2} \ar{r}{} \&
                \mathcal{L}_{n - 2} \ar{r}{} \&
                    \M_{n - 2} \ar{r} \&
                        \M_{n - 2}[1]
        \end{tikzcd}
    \end{align*}
    provides the indecomposable matrix factorization $\M_{n - 3}$
    after the trivial components and $\M_{n - 1}$ are split off from $\mathcal{L}_{n - 2}$.
    Repeating similar procedures yields the indecomposable matrix factorizations
    $\M_{n - 2}, \M_{n - 3}, \ldots, \M_{3}$.

    \textit{Step 3.} We produce $\M_n$ from $\M_3$.
    Construct the Auslander--Reiten triangle
    \begin{align*}
        \begin{tikzcd}[ampersand replacement=\&]
            \M_{3} \ar{r}{} \&
                \mathcal{L}_{3} \ar{r} \&
                    \M_{3} \ar{r} \&
                        \M_{3}[1]
        \end{tikzcd}
    \end{align*}
    and split off the trivial components, $\M_{2}$ and  $\M_{4}$ from $\mathcal{L}_{3}$.
    Then we obtain the indecomposable matrix factorization $\M_{n}$.
\end{proof}

A similar argument suggests that analogous statements hold for the other types.
In what follows, we shall prove that this is indeed true.
Let $\pi \colon \widetilde{X} \to \Spec R$ be the minimal resolution.
We identify maximal Cohen--Macaulay $R$-modules with the associated coherent sheaves, and set
\begin{align*}
    &\widetilde{M} \coloneqq (\pi^*M)^{\vee \vee}
        & &\text{for $M \in \MCM(R)$.}
\end{align*}

\begin{lem}
    \label{lem:IdealSh}
    Let $M \subset R^r$ be a maximal Cohen--Macaulay $R$-module with $\rk_R M = r$. Then
    $\det (\widetilde{M})$ is a locally free ideal sheaf of $\strsh_{\widetilde{X}}$.
\end{lem}

\begin{proof}
    As $\T \coloneqq \Ker (\pi^* (M \hookrightarrow R^r))$ is a torsion sheaf,
    applying the functor $(-)^{\vee \vee}$ to the exact sequence
    \begin{align*}
        \begin{tikzcd}[ampersand replacement=\&]
            0 \ar{r} \&
                \T \ar{r} \&
                    \pi^* M \ar{r} \&
                        \strsh_{\widetilde{X}}^{\oplus r}
        \end{tikzcd}
    \end{align*}
    tells us the canonical morphism $\widetilde{M} \to \strsh_{\widetilde{X}}^{\oplus r}$ is injective.
    Since, by $\rk_R M = r$, this morphism is isomorphic at the generic point $\eta$ of $\widetilde{X}$,
    so is the morphism $\det(\widetilde{M}) \to \strsh_{\widetilde{X}}$. Therefore,
    \begin{equation}
        \Ker (\det(\widetilde{M}) \to \strsh_{\widetilde{X}})_{\eta} = 0. \label{eq:AtGenPt}
    \end{equation} 
    Considering that $\widetilde{X}$ is integral and $\det(\widetilde{M})$ is torsion free,
    \cref{eq:AtGenPt} implies
    \[  \Ker (\det(\widetilde{M}) \to \strsh_{\widetilde{X}}) = 0. \qedhere \]
\end{proof}

The following lemma is straightforward.

\begin{lem}
    \label{lem:MCMRank}
    \begin{items}
        \item \label{lem:MCMRank:1} Let $\M_1, \M_2, \ldots, \M_n$ be matrix factorizations of type $A_n$ in \cref{subsec:A_n},
        and $M_1, M_2, \allowbreak \ldots, M_n$ maximal Cohen--Macaulay $R$-modules
        induced from them by \cref{lem:SizeDown} \cref*{lem:SizeDown:1}. Then
        \begin{align*}
            &\rk_R M_i = 1
                & &\text{for $i = 1, 2, \ldots, n$.}
        \end{align*}
        \item \label{lem:MCMRank:2} Let $\M_1, \M_2, \ldots, \M_{2n}$ be matrix factorizations of type $D_{2n}^r$
        in \cref{subsec:D_even^r}, and $M_1, M_2, \ldots, M_{2n}$ maximal Cohen--Macaulay $R$-modules
        induced from them by \cref{lem:SizeDown} \cref*{lem:SizeDown:2}. Then
        \[
            \rk_R M_i =
            \begin{cases}
                1 & \text{for $i = 1, 2n - 1, 2n$,} \\
                2 & \text{for $i = 2, 3, \ldots, 2n - 2$.}
            \end{cases}
        \]
        \item \label{lem:MCMRank:3} Let $\M_1, \M_2, \ldots, \M_{2n + 1}$ be matrix factorizations of type $D_{2n + 1}^r$
        in \cref{subsec:D_odd^r}, and $M_1, M_2, \ldots, M_{2n + 1}$ maximal Cohen--Macaulay $R$-modules
        induced from them by \cref{lem:SizeDown} \cref*{lem:SizeDown:2}. Then
        \[
            \rk_R M_i =
            \begin{cases}
                1 & \text{for $i = 1, 2n, 2n + 1$,} \\
                2 & \text{for $i = 2, 3, \ldots, 2n - 1$.}
            \end{cases}
        \]
    \end{items}
\end{lem}

In each case \cref*{lem:MCMRank:1}, \cref*{lem:MCMRank:2} and \cref*{lem:MCMRank:3} of \cref{lem:MCMRank},
since any $M_i$ is embedded in $R^{r_i}$ where $r_i \coloneqq \rk_R (M_i)$,
\cref{lem:IdealSh} asserts that there exists a closed subscheme $Z_i$ of $\widetilde{X}$ such that
$\det (\widetilde{M_i}) \cong \strsh_{\widetilde{X}}(-Z_i)$. 
With this notation, $\exset{Z_i}_i$ is given by the following proposition.

\begin{prop}
    \label{prop:HardWork}
    \begin{items}
        \item \label{prop:HardWork:1} Assume the rational double point $(X, p)$ is of type $A_n$.
        Then in the divisor class group $\Cl (\widetilde{X})$,
        \begin{align*}
            \begin{pmatrix}
                Z_1 & Z_2 & \cdots & Z_n
            \end{pmatrix}
            = L_n
            \begin{pmatrix}
                1 & 2 & \cdots & n
            \end{pmatrix}
            + \begin{pmatrix}
                E_1 & E_2 & \cdots & E_n
            \end{pmatrix}
            \begin{pmatrix}
                1 & 1 & 1 & \cdots & 1 \\
                1 & 2 & 2 & \cdots & 2 \\
                1 & 2 & 3 & \cdots & 3 \\
                \vdots & \vdots & \vdots & \ddots & \vdots \\
                1 & 2 & 3 & \cdots & n
            \end{pmatrix},
        \end{align*}
        where $L_n$ is the strict transform of the vanishing locus $\VV[R](y, z)$ and
        the exceptional prime divisors are illustrated in the dual graph
        \[ 
            \begin{tikzcd}[ampersand replacement=\&]
                E_1 \ar[no head]{r} \&
                    E_2 \ar[no head]{r} \&
                        \cdots \ar[no head]{r} \&
                            E_n \ar[no head]{r} \&
                                L_{n}\rlap{.}
            \end{tikzcd}
        \]
        \item \label{prop:HardWork:2} Assume the rational double point $(X, p)$ is of type $D_{2n}^r$. Define
        \begin{align*}
            &U = \begin{pmatrix}
                \bm{u}_1 & \bm{u}_2 & \cdots & \bm{u}_{2n - 3}
            \end{pmatrix}
            \in M_{3, 2n - 3} (\ZZ), \\
            &V = \begin{pmatrix}
                \bm{v}_1 & \bm{v}_2 & \cdots & \bm{v}_{2n - 3}
            \end{pmatrix}
            \in M_{2n - 2, 2n - 3} (\ZZ), \\
            &W = \begin{pmatrix}
                \bm{w}_1 & \bm{w}_2 & \cdots & \bm{w}_{2n - 3}
            \end{pmatrix}
            \in M_{2, 2n - 3} (\ZZ)
        \end{align*}
        by setting for $j = 1, 2, \ldots, 2n - 3$,
        \begin{align*}
            &\bm{u}_j \coloneqq
            \begin{cases}
                \trans{\begin{pmatrix}
                    1 & 1 & 1
                \end{pmatrix}}
                    & \text{if $j$ is odd,} \\
                \trans{\begin{pmatrix}
                    0 & 1 & 1
                \end{pmatrix}}
                    & \text{if $j$ is even,}
            \end{cases} \\
            &\bm{v}_j \coloneqq
            \begin{cases}
                \trans{\begin{pNiceArray}{cccc:cccc}[margin]
                    2 & 4 & \cdots & 2j + 2 & 2j + 3 & 2j + 4 & \cdots & 2n + j - 1
                \end{pNiceArray}}
                    & \text{if $j$ is even,} \\
                \trans{\begin{pNiceArray}{cccc}[margin]
                    3 & 5 & \cdots & 4n - 3
                \end{pNiceArray}}
                    & \text{if $j = 2n - 3$,} \\
                \trans{\begin{pNiceArray}{cccc:cccc}[margin]
                    3 & 5 & \cdots & 2j + 3 & 2j + 4 & 2j + 5 & \cdots & 2n + j
                \end{pNiceArray}}
                    & \text{otherwise,}
            \end{cases} \\
            &\bm{w}_j \coloneqq
            \trans{\begin{pmatrix}
                n + \ceil{\dfrac{j}{2}} & n + \ceil{\dfrac{j}{2}}
            \end{pmatrix}}. \\
        \end{align*}
        Then in the divisor class group $\Cl (\widetilde{X})$,
        \begin{align*}
            \begin{pmatrix}
                Z_1 & Z_2 & \cdots & Z_{2n}
            \end{pmatrix}
            =
            &\begin{pmatrix}
                L_1 & L_{2n - 1} & L_{2n}
            \end{pmatrix}
            \begin{pNiceArray}{c:ccc:cc}[margin]
                1 & \Block{3-3}{U} & & & 0 & 1 \\
                0 & & & & 1 & 1 \\
                0 & & & & 0 & 0
            \end{pNiceArray} \\
            &+ \begin{pmatrix}
                E_1 & E_2 & \cdots & E_{2n}
            \end{pmatrix}
            \begin{pNiceArray}{c:cccc:cc}[margin]
                2 & \Block{4-4}{V} & & & & 1 & 2 \\
                2 & & & & & 2 & 3 \\
                \vdots & & & & & \vdots & \vdots \\
                2 & & & & & 2n - 2 & 2n - 1 \\
                \hdottedline
                1 & \Block{2-4}{W} & & & & n & n \\
                1 & & & & & n - 1 & n
            \end{pNiceArray}, \\
        \end{align*}
        where $L_1$, $L_{2n - 1}$ and $L_{2n}$
        are the strict transforms of the vanishing loci
        $\VV[R](y, z)$, $\VV[R](x, z)$ and $\VV[R](x + y^{n - 1}, z)$ respectively,
        and the exceptional prime divisors are illustrated in the dual graph
        
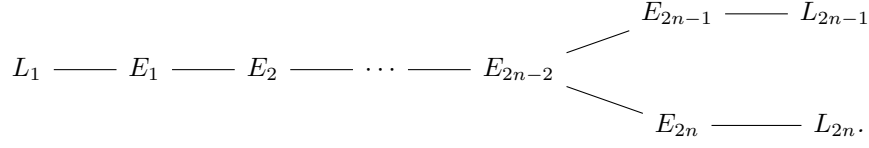
\[ 
            \begin{tikzcd}[ampersand replacement=\&, row sep=tiny]
                \&
                    \&
                        \&
                            \&
                                \&
                                    E_{2n - 1} \ar[no head]{r} \&
                                        L_{2n - 1}\\
                L_1 \ar[no head]{r} \&
                    E_1 \ar[no head]{r} \&
                        E_2 \ar[no head]{r} \&
                            \cdots \ar[no head]{r} \&
                                E_{2n - 2} \ar[no head]{ur} \ar[no head]{dr} \&
                                    \&
                                        \\
                \&
                    \&
                        \&
                            \&
                                \&
                                    E_{2n} \ar[no head]{r} \&
                                        L_{2n}\rlap{.}
            \end{tikzcd}
        \]
        \item \label{prop:HardWork:3} Assume the rational double point $(X, p)$ is of type $D_{2n + 1}^r$. Define
        \begin{align*}
            &U = \begin{pmatrix}
                \bm{u}_1 & \bm{u}_2 & \cdots & \bm{u}_{2n - 2}
            \end{pmatrix}
            \in M_{2, 2n - 2} (\ZZ), \\
            &V = \begin{pmatrix}
                \bm{v}_1 & \bm{v}_2 & \cdots & \bm{v}_{2n - 2}
            \end{pmatrix}
            \in M_{2n - 1, 2n - 2} (\ZZ), \\
            &W = \begin{pmatrix}
                \bm{w}_1 & \bm{w}_2 & \cdots & \bm{w}_{2n - 2}
            \end{pmatrix}
            \in M_{2, 2n - 2} (\ZZ)
        \end{align*}
        by setting for $j = 1, 2, \ldots, 2n - 2$,
        \begin{align*}
            &\bm{u}_j \coloneqq
            \begin{cases}
                \trans{\begin{pmatrix}
                    1 & 2
                \end{pmatrix}}
                    & \text{if $j$ is odd,} \\
                \trans{\begin{pmatrix}
                    0 & 2
                \end{pmatrix}}
                    & \text{if $j$ is even,}
            \end{cases} \\
            &\bm{v}_j \coloneqq
            \begin{cases}
                \trans{\begin{pNiceArray}{cccc:cccc}[margin]
                    3 & 5 & \cdots & 2j + 3 & 2j + 4 & 2j + 5 & \cdots & 2n + j + 1
                \end{pNiceArray}}
                    & \text{if $j$ is odd,} \\
                \trans{\begin{pNiceArray}{cccc}[margin]
                    2 & 4 & \cdots & 4n - 2
                \end{pNiceArray}}
                    & \text{if $j = 2n - 2$,} \\
                \trans{\begin{pNiceArray}{cccc:cccc}[margin]
                    2 & 4 & \cdots & 2j + 2 & 2j + 3 & 2j + 4 & \cdots & 2n + j
                \end{pNiceArray}}
                    & \text{otherwise,}
            \end{cases} \\
            &\bm{w}_j \coloneqq
            \trans{\begin{pmatrix}
                n + \ceil{\dfrac{j}{2}} & n + \ceil{\dfrac{j}{2}} + 1
            \end{pmatrix}}. \\
        \end{align*}
        Then in the divisor class group $\Cl (\widetilde{X})$,
        \begin{align*}
            \begin{pmatrix}
                Z_1 & Z_2 & \cdots & Z_{2n + 1}
            \end{pmatrix}
            =
            &\begin{pmatrix}
                L_1 & L_{2n + 1}
            \end{pmatrix}
            \begin{pNiceArray}{c:cc:cc}[margin]
                1 & \Block{2-2}{U} & & 0 & 1 \\
                0 & & & 1 & 1
            \end{pNiceArray} \\
            &+ \begin{pmatrix}
                E_1 & E_2 & \cdots & E_{2n + 1}
            \end{pmatrix}
            \begin{pNiceArray}{c:cccc:cc}[margin]
                2 & \Block{4-4}{V} & & & & 1 & 2 \\
                2 & & & & & 2 & 3 \\
                \vdots & & & & & \vdots & \vdots \\
                2 & & & & & 2n - 1 & 2n \\
                \hdottedline
                1 & \Block{2-4}{W} & & & & n & n \\
                1 & & & & & n & n + 1
            \end{pNiceArray}, \\
        \end{align*}
        where $L_1$ and $L_{2n + 1}$
        are the strict transforms of the vanishing loci $\VV[R](y, z)$ and $\VV[R](x, z)$ respectively,
        and the exceptional prime divisors are illustrated in the dual graph
        
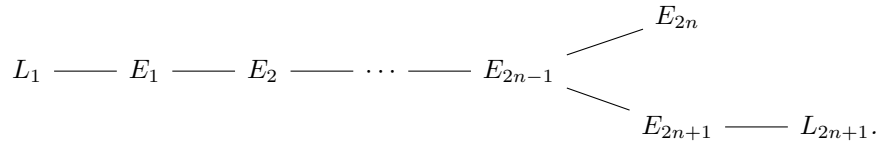
\[ 
            \begin{tikzcd}[ampersand replacement=\&, row sep=tiny]
                \&
                    \&
                        \&
                            \&
                                \&
                                    E_{2n} \&
                                        \\
                L_1 \ar[no head]{r} \&
                    E_1 \ar[no head]{r} \&
                        E_2 \ar[no head]{r} \&
                            \cdots \ar[no head]{r} \&
                                E_{2n - 1} \ar[no head]{ur} \ar[no head]{dr} \&
                                    \&
                                        \\
                \&
                    \&
                        \&
                            \&
                                \&
                                    E_{2n + 1} \ar[no head]{r} \&
                                        L_{2n + 1}\rlap{.}
            \end{tikzcd}
        \]
    \end{items}
\end{prop}

\begin{proof}
    Take a maximal Cohen--Macaulay $R$-module $M \coloneqq M_i$
    and a prime divisor $C = \overline{\exset{\eta}}$ on $\widetilde{X}$.
    Let the morphism $(\pi^* (M \hookrightarrow R^{r_i}))_{\eta}$ be denoted by
    $\phi \colon M \tens[R] \strsh_{\widetilde{X}, \eta} \to \strsh_{\widetilde{X}, \eta}^{\oplus r_i}$,
    which is isomorphic at the generic point of $\widetilde{X}$.
    Note that the function field of $\widetilde{X}$ is identified with 
    the fraction field $K(R)$ of $R$. Then
    \begin{align*}
        \widetilde{M}_{\eta}
        &\cong ((\pi^* M)_{\eta})^{\vee \vee} \\
        &\cong \dfrac{M \tens[R] \strsh_{\widetilde{X}, \eta}}
        {\Tor_{\strsh_{\widetilde{X}, \eta}} (M \tens[R] \strsh_{\widetilde{X}, \eta})} \\
        &= \dfrac{M \tens[R] \strsh_{\widetilde{X}, \eta}}
        {\Ker \expar{M \tens[R] \strsh_{\widetilde{X}, \eta} \to M \tens[R] K(R)}} \\
        &= \dfrac{M \tens[R] \strsh_{\widetilde{X}, \eta}}{\Ker \phi} \\
        &\cong \Im \phi.
    \end{align*}
    Therefore,
    \[
        \det (\widetilde{M})_{\eta}
        \cong \det (\widetilde{M}_{\eta})
        \cong \det (\Im \phi)
        \cong \Im (\det M \to R) \strsh_{\widetilde{X}, \eta}.
    \]
    This implies that the multiplicity of $C$ in $Z_i$ coincides with the valuation of the ideal
    $\Im (\det M \to R) \strsh_{\widetilde{X}, \eta}$.
    In particular, it follows that any prime divisor appearing in the cycle $Z_i$ lies in
    the inverse image of the vanishing locus $\VV[R](\Im(\det M \to R))$ under $\pi \colon \widetilde{X} \to \Spec R$.
    By calculating the multiplicities of such prime divisors directly, we obtain the desired results.
\end{proof}

In each case \cref*{prop:HardWork:1}, \cref*{prop:HardWork:2} and \cref*{prop:HardWork:3} of \cref{prop:HardWork},
we can verify
\begin{align*}
    &c_1(\widetilde{M_i}) \cdot E_j = -Z_i \cdot E_j = \delta_{i, j}
        & &\text{for any $M_i$ and $E_j$,}
\end{align*}
and hence $\exset{M_i}_i$ is a complete set of pairwise non-isomorphic indecomposable objects in $\sMCM(R)$.
Combining this with \cref{thm:IrrConfig} yields our second main theorem.

\begin{thm}
    \label{thm:Main2}
    \begin{items}
        \item The matrix factorizations $\M_1, \M_2, \ldots, \M_n$ of type $A_n$ in \cref{subsec:A_n}
        constitute a complete set of pairwise non-isomorphic indecomposable objects in $\HMF_S (f)$ of type $A_n$,
        and the Auslander--Reiten quiver is given by
        
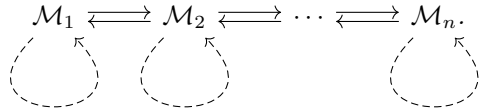
\[ 
\begin{tikzcd}[ampersand replacement=\&]
    \M_1 \ar[shift left=0.4ex]{r} \ar[out=225,in=315,loop, dashed] \&
        \M_2 \ar[shift left=0.4ex]{r} \ar[shift left=0.4ex]{l} \ar[out=225,in=315,loop, dashed] \&
            \cdots \ar[shift left=0.4ex]{r} \ar[shift left=0.4ex]{l} \&
                \M_n\rlap{.} \ar[shift left=0.4ex]{l} \ar[out=225,in=315,loop, dashed]
\end{tikzcd}
 \]
        \item The matrix factorizations $\M_1, \M_2, \ldots, \M_{2n}$ of type $D_{2n}^r$ in \cref{subsec:D_even^r}
        constitute a complete set of pairwise non-isomorphic indecomposable objects in $\HMF_S (f)$ of type $D_{2n}^r$,
        and the Auslander--Reiten quiver is given by
        
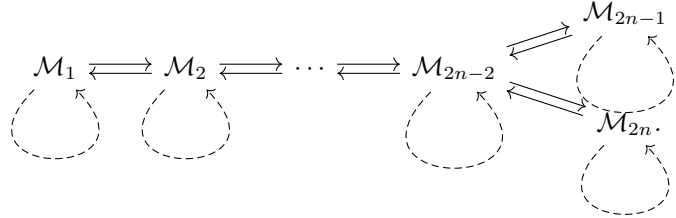
\[ 
\begin{tikzcd}[ampersand replacement=\&, row sep=tiny]
    \&
        \&
            \&
                \&
                    \M_{2n - 1} \ar[shift left=0.4ex]{dl} \ar[out=225,in=315,loop, dashed] \\
    \M_1 \ar[shift left=0.4ex]{r} \ar[out=225,in=315,loop, dashed] \&
        \M_2 \ar[shift left=0.4ex]{r} \ar[shift left=0.4ex]{l} \ar[out=225,in=315,loop, dashed] \&
            \cdots \ar[shift left=0.4ex]{r} \ar[shift left=0.4ex]{l} \&
                \M_{2n - 2} \ar[shift left=0.4ex]{l} \ar[shift left=0.4ex]{ur} \ar[shift left=0.4ex]{dr} \ar[out=225,in=315,loop, dashed] \&
                    \\
    \&
        \&
            \&
                \&
                    \M_{2n}\rlap{.} \ar[shift left=0.4ex]{ul} \ar[out=225,in=315,loop, dashed]
\end{tikzcd}
 \]
        \item The matrix factorizations $\M_1, \M_2, \ldots, \M_{2n + 1}$ of type $D_{2n + 1}^r$ in \cref{subsec:D_odd^r}
        constitute a complete set of pairwise non-isomorphic indecomposable objects in $\HMF_S (f)$ of type $D_{2n + 1}^r$,
        and the Auslander--Reiten quiver is given by
        
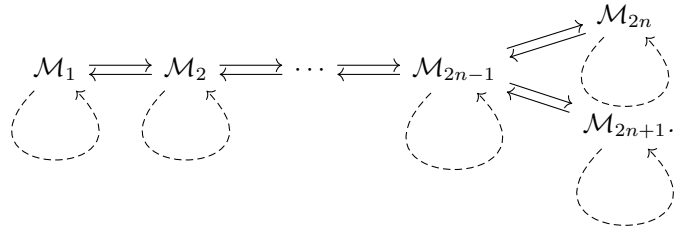
\[ 
\begin{tikzcd}[ampersand replacement=\&, row sep=tiny]
    \&
        \&
            \&
                \&
                    \M_{2n} \ar[shift left=0.4ex]{dl} \ar[out=225,in=315,loop, dashed] \\
    \M_1 \ar[shift left=0.4ex]{r} \ar[out=225,in=315,loop, dashed] \&
        \M_2 \ar[shift left=0.4ex]{r} \ar[shift left=0.4ex]{l} \ar[out=225,in=315,loop, dashed] \&
            \cdots \ar[shift left=0.4ex]{r} \ar[shift left=0.4ex]{l} \&
                \M_{2n - 1} \ar[shift left=0.4ex]{l} \ar[shift left=0.4ex]{ur} \ar[shift left=0.4ex]{dr} \ar[out=225,in=315,loop, dashed] \&
                    \\
    \&
        \&
            \&
                \&
                    \M_{2n + 1}\rlap{.} \ar[shift left=0.4ex]{ul} \ar[out=225,in=315,loop, dashed]
\end{tikzcd}
 \]
    \end{items}
\end{thm}

\begingroup

\AtBeginEnvironment{pNiceArray}{
    \setlength{\arraycolsep}{3pt}
    \renewcommand{\arraystretch}{0.8}
}

\section{Lists of the indecomposable objects}
\label{sec:Lists}
Let types $D_n$ $(n \geq 4)$, $E_6$, $E_7$ and $E_8$ in \cref{tab:RDPEq} also be denoted by
$D_n^0$, $E_6^0$, $E_7^0$ and $E_8^0$ respectively.
\subsection{\texorpdfstring{$A_n: f = z^{n + 1} + xy$}{An}}
\label{subsec:A_n}
For $i = 1, 2, \ldots, n$,
\begin{align*}
    \M_i \coloneqq \mbox{\footnotesize $\expar{
        \begin{pmatrix}
            z^{n + 1 - i} & -y \\
            x & z^i
        \end{pmatrix},
        \begin{pmatrix}
            z^i & y \\
            -x & z^{n + 1 - i}
        \end{pmatrix}
    }$}.
\end{align*}

\subsection{\texorpdfstring{$D_{2n}^r$: $f = z^{2} + x^2y + xy^n + (1 - \delta_{0, r}) xy^{n - r}z$}{Devenr}}
\label{subsec:D_even^r}
Let $\varepsilon$ denote $(1 - \delta_{0, r})$.
\begin{align*}
    &\M_1 \coloneqq \mbox{\footnotesize $\expar{
        \begin{pmatrix}
            z & x^2 + xy^{n - 1} \\
            -y & z + \varepsilon xy^{n - r}
        \end{pmatrix},
        \begin{pmatrix}
            z + \varepsilon xy^{n - r} & -x^2 - xy^{n - 1} \\
            y & z
        \end{pmatrix}
    }$}, \\
    &\M_{2n - 1} \coloneqq \mbox{\footnotesize $\expar{
        \begin{pmatrix}
            z & xy + y^n \\
            -x & z + \varepsilon xy^{n - r}
        \end{pmatrix},
        \begin{pmatrix}
            z + \varepsilon xy^{n - r} & -xy - y^n \\
            x & z
        \end{pmatrix}
    }$}, \\
    &\M_{2n} \coloneqq \mbox{\footnotesize $\expar{
        \begin{pmatrix}
            z & x + y^{n - 1} \\
            -xy & z + \varepsilon xy^{n - r}
        \end{pmatrix},
        \begin{pmatrix}
            z + \varepsilon xy^{n - r} & -x - y^{n - 1} \\
            xy & z
        \end{pmatrix}
    }$}.
\end{align*}
For $i = 1, 2, \ldots, n - 1$,
\[
    \M_{2i} \coloneqq \mbox{\footnotesize $\expar{
            \begin{pNiceArray}{cc:w{c}{5em}w{c}{2em}}[margin]
                \Block{2-2}<\small \boldmath>{z} & & xy & y^i \\
                & & -xy^{n - i} & x \\
                \hdottedline
                -x & y^i & \Block{2-2}<\small \boldmath>{z + \varepsilon xy^{n - r}} & \\
                -xy^{n - i} & -xy & &
            \end{pNiceArray},
            \begin{pNiceArray}{w{c}{3em}w{c}{4em}:cc}[margin]
                \Block{2-2}<\small \boldmath>{z + \varepsilon xy^{n - r}} & & -xy & -y^i \\
                & & xy^{n - i} & -x \\
                \hdottedline
                x & -y^i & \Block{2-2}<\small \boldmath>{z} & \\
                xy^{n - i} & xy & &
            \end{pNiceArray}
    }$}.
\]
For $i = 1, 2, \ldots, n - 2$,
\[
    \M_{2i + 1} \coloneqq \mbox{\footnotesize $\expar{
        \begin{pNiceArray}{cc:w{c}{4em}w{c}{3em}}[margin]
            \Block{2-2}<\small \boldmath>{z} & & xy & xy^{n - i} \\
            & & -y^{i + 1} & xy \\
            \hdottedline
            -x & xy^{n - i - 1} & \Block{2-2}<\small \boldmath>{z + \varepsilon xy^{n - r}} & \\
            -y^i & -x & &
        \end{pNiceArray},
        \begin{pNiceArray}{w{c}{1em}w{c}{6em}:cc}[margin]
            \Block{2-2}<\small \boldmath>{z + \varepsilon xy^{n - r}} & & -xy & -xy^{n - i} \\
            & & y^{i + 1} & -xy \\
            \hdottedline
            x & -xy^{n - i - 1} & \Block{2-2}<\small \boldmath>{z} & \\
            y^i & x & &
        \end{pNiceArray}
    }$}.
\]

\subsection{\texorpdfstring{$D_{2n + 1}^r$: $f = z^{2} + x^2y + y^nz + (1 - \delta_{0, r}) xy^{n - r}z$}{Doddr}}
\label{subsec:D_odd^r}
Let $\varepsilon$ denote $(1 - \delta_{0, r})$.
\begin{align*}
    &\M_1 \coloneqq \mbox{\footnotesize $\expar{
        \begin{pmatrix}
            z & x^2 \\
            -y & z + y^n + \varepsilon xy^{n - r}
        \end{pmatrix},
        \begin{pmatrix}
            z + y^n + \varepsilon xy^{n - r} & -x^2 \\
            y & z
        \end{pmatrix}
    }$}, \\
    &\M_{2n} \coloneqq \mbox{\footnotesize $\expar{
        \begin{pmatrix}
            z & xy \\
            -x & z + y^n + \varepsilon xy^{n - r}
        \end{pmatrix},
        \begin{pmatrix}
            z + y^n + \varepsilon xy^{n - r} & -xy \\
            x & z
        \end{pmatrix}
    }$}, \\
    &\M_{2n + 1} \coloneqq \mbox{\footnotesize $\expar{
        \begin{pmatrix}
            z & x \\
            -xy & z + y^n + \varepsilon xy^{n - r}
        \end{pmatrix},
        \begin{pmatrix}
            z + y^n + \varepsilon xy^{n - r} & -x \\
            xy & z
        \end{pmatrix}
    }$}.
\end{align*}
For $i = 1, 2, \ldots, n$,
\begin{align*}
    &\M_{2i} \coloneqq \mbox{\footnotesize $\expar{
            \begin{pNiceArray}{cc:w{c}{5em}w{c}{4em}}[margin]
                \Block{2-2}<\small \boldmath>{z} & & xy & y^{i} \\
                & & 0 & x \\
                \hdottedline
                -x & y^i & \Block{2-2}<\small \boldmath>{z + y^n + \varepsilon xy^{n - r}} & \\
                0 & -xy & &
            \end{pNiceArray},
            \begin{pNiceArray}{w{c}{4em}w{c}{5em}:cc}[margin]
                \Block{2-2}<\small \boldmath>{z + y^n + \varepsilon xy^{n - r}} & & -xy & -y^i \\
                & & 0 & -x \\
                \hdottedline
                x & -y^i & \Block{2-2}<\small \boldmath>{z} & \\
                0 & xy & &
            \end{pNiceArray}
    }$}, \\
    &\M_{2i + 1} \coloneqq \mbox{\footnotesize $\expar{
        \begin{pNiceArray}{cc:w{c}{5em}w{c}{4em}}[margin]
            \Block{2-2}<\small \boldmath>{z} & & xy & 0 \\
            & & -y^{i + 1} & xy \\
            \hdottedline
            -x & 0 & \Block{2-2}<\small \boldmath>{z + y^n + \varepsilon xy^{n - r}} & \\
            -y^i & -x & &
        \end{pNiceArray},
        \begin{pNiceArray}{w{c}{4em}w{c}{5em}:cc}[margin]
            \Block{2-2}<\small \boldmath>{z + y^n + \varepsilon xy^{n - r}} & & -xy & 0 \\
            & & y^{i + 1} & -xy \\
            \hdottedline
            x & 0 & \Block{2-2}<\small \boldmath>{z} & \\
            y^i & x & &
        \end{pNiceArray}
    }$}.
\end{align*}

\subsection{\texorpdfstring{$E_6^r$: $f = z^{2} + x^3 + y^2z + (1 - \delta_{0, r}) xyz$}{E6r}}
\label{subsec:E_6^r}
Let $\varepsilon$ denote $(1 - \delta_{0, r})$.
\begin{align*}
    &\M_{1} \coloneqq \mbox{\footnotesize $\expar{
        \begin{pmatrix}
            z & x^2 \\
            -x & z + y^2 + \varepsilon xy
        \end{pmatrix},
        \begin{pmatrix}
            z + y^2 + \varepsilon xy & -x^2 \\
            x & z
        \end{pmatrix}
    }$}, \\
    &\M_{2} \coloneqq \mbox{\footnotesize $\expar{
        \begin{pNiceArray}{cc:w{c}{3.5em}w{c}{3.5em}}[margin]
            \Block{2-2}<\small \boldmath>{z} & & x^2 & 0 \\
            & & xy & -x^2 \\
            \hdottedline
            -x & 0 & \Block{2-2}<\small \boldmath>{z + y^2 + \varepsilon xy} & \\
            -y & x & &
        \end{pNiceArray},
        \begin{pNiceArray}{w{c}{2.5em}w{c}{4.5em}:cc}[margin]
            \Block{2-2}<\small \boldmath>{z + y^2 + \varepsilon xy} & & -x^2 & 0 \\
            & & -xy & x^2 \\
            \hdottedline
            x & 0 & \Block{2-2}<\small \boldmath>{z} & \\
            y & -x & &
        \end{pNiceArray}
    }$}, \\
    &\M_{3} \coloneqq \mbox{\footnotesize $\expar{
        \begin{pNiceArray}{ccc:w{c}{2em}w{c}{2em}w{c}{2em}}[margin]
            \Block{3-3}<\normalsize \boldmath>{z} & & & x^2 & y^2 & xy \\
            & & & 0 & x^2 & 0 \\
            & & & 0 & -xy & -x^2 \\
            \hdottedline
            -x & 0 & -y & \Block{3-3}<\normalsize \boldmath>{z + y^2 + \varepsilon xy} & & \\
            0 & -x & 0 & & & \\
            0 & y & x & & &
        \end{pNiceArray},
        \begin{pNiceArray}{w{c}{1em}w{c}{2em}w{c}{3em}:ccc}[margin]
            \Block{3-3}<\normalsize \boldmath>{z + y^2 + \varepsilon xy} & & & -x^2 & -y^2 & -xy \\
            & & & 0 & -x^2 & 0 \\
            & & & 0 & xy & x^2 \\
            \hdottedline
            x & 0 & y & \Block{3-3}<\normalsize \boldmath>{z} & & \\
            0 & x & 0 & & & \\
            0 & -y & -x & & &
        \end{pNiceArray}
    }$}, \\
    &\M_{4} \coloneqq \mbox{\footnotesize $\expar{
        \begin{pNiceArray}{cc:w{c}{4em}w{c}{3em}}[margin]
            \Block{2-2}<\small \boldmath>{z} & & x & 0 \\
            & & -y & x \\
            \hdottedline
            -x^2 & 0 & \Block{2-2}<\small \boldmath>{z + y^2 + \varepsilon xy} & \\
            -xy & -x^2 & &
        \end{pNiceArray},
        \begin{pNiceArray}{w{c}{3em}w{c}{4em}:cc}[margin]
            \Block{2-2}<\small \boldmath>{z + y^2 + \varepsilon xy} & & -x & 0 \\
            & & y & -x \\
            \hdottedline
            x^2 & 0 & \Block{2-2}<\small \boldmath>{z} & \\
            xy & x^2 & &
        \end{pNiceArray}
    }$}, \\
    &\M_{5} \coloneqq \mbox{\footnotesize $\expar{
        \begin{pmatrix}
            z & x \\
            -x^2 & z + y^2 + \varepsilon xy
        \end{pmatrix},
        \begin{pmatrix}
            z + y^2 + \varepsilon xy & -x \\
            x^2 & z
        \end{pmatrix}
    }$}, \\
    &\M_{6} \coloneqq \mbox{\footnotesize $\expar{
        \begin{pNiceArray}{cc:w{c}{3.5em}w{c}{3.5em}}[margin]
            \Block{2-2}<\small \boldmath>{z} & & x^2 & y \\
            & & 0 & -x \\
            \hdottedline
            -x & -y & \Block{2-2}<\small \boldmath>{z + y^2 + \varepsilon xy} & \\
            0 & x^2 & &
        \end{pNiceArray},
        \begin{pNiceArray}{w{c}{2.5em}w{c}{4.5em}:cc}[margin]
            \Block{2-2}<\small \boldmath>{z + y^2 + \varepsilon xy} & & -x^2 & -y \\
            & & 0 & x \\
            \hdottedline
            x & y & \Block{2-2}<\small \boldmath>{z} & \\
            0 & -x^2 & &
        \end{pNiceArray}
    }$}.
\end{align*}

\subsection{\texorpdfstring{$E_{7}^0$: $f = z^2 + x^3 + xy^3$}{E70}}
\label{subsec:E_7^0}
\begin{align*}
    &\M_{1} \coloneqq \mbox{\footnotesize $\expar{
        \begin{pNiceArray}{cc:cc}[margin]
            \Block{2-2}<\small \boldmath>{z} & & x^2 & xy^2 \\
            & & y & -x \\
            \hdottedline
            -x & -xy^2 & \Block{2-2}<\small \boldmath>{z} & \\
            -y & x^2 & &
        \end{pNiceArray},
        \begin{pNiceArray}{cc:cc}[margin]
            \Block{2-2}<\small \boldmath>{z} & & -x^2 & -xy^2 \\
            & & -y & x \\
            \hdottedline
            x & xy^2 & \Block{2-2}<\small \boldmath>{z} & \\
            y & -x^2 & &
        \end{pNiceArray}
    }$}, \\
    &\M_{2} \coloneqq \mbox{\footnotesize $\expar{
        \begin{pNiceArray}{ccc:ccc}[margin]
            \Block{3-3}<\normalsize \boldmath>{z} & & & x^2 & xy^2 & -x^2y \\
            & & & -xy & x^2 & xy^2 \\
            & & & y^2 & -xy & x^2 \\
            \hdottedline
            -x & 0 & -xy & \Block{3-3}<\normalsize \boldmath>{z} & & \\
            -y & -x & 0 & & & \\
            0 & -y & -x & & &
        \end{pNiceArray},
        \begin{pNiceArray}{ccc:ccc}[margin]
            \Block{3-3}<\normalsize \boldmath>{z} & & & -x^2 & -xy^2 & x^2y \\
            & & & xy & -x^2 & -xy^2 \\
            & & & -y^2 & xy & -x^2 \\
            \hdottedline
            x & 0 & xy & \Block{3-3}<\normalsize \boldmath>{z} & & \\
            y & x & 0 & & & \\
            0 & y & x & & &
        \end{pNiceArray}
    }$}, \\
    &\M_{3} \coloneqq \mbox{\footnotesize $\expar{
        \begin{pNiceArray}{cccc:cccc}[margin] 
            \Block{4-4}<\large \boldmath \boldmath>{z} & & & & 0 & 0 & x^2 & xy^2 \\
            & & & & 0 & 0 & xy & -x^2 \\
            & & & & x & y^2 & 0 & -xy \\
            & & & & y & -x & x & 0 \\
            \hdottedline
            0 & xy & -x^2 & -xy^2 & \Block{4-4}<\large \boldmath \boldmath>{z} & & & \\
            -x & 0 & -xy & x^2 & & & & \\
            -x & -y^2 & 0 & 0 & & & & \\
            -y & x & 0 & 0 & & & &
        \end{pNiceArray},
        \begin{pNiceArray}{cccc:cccc}[margin]
            \Block{4-4}<\large \boldmath \boldmath>{z} & & & & 0 & 0 & -x^2 & -xy^2 \\
            & & & & 0 & 0 & -xy & x^2 \\
            & & & & -x & -y^2 & 0 & xy \\
            & & & & -y & x & -x & 0 \\
            \hdottedline
            0 & -xy & x^2 & xy^2 & \Block{4-4}<\large \boldmath \boldmath>{z} & & & \\
            x & 0 & xy & -x^2 & & & & \\
            x & y^2 & 0 & 0 & & & & \\
            y & -x & 0 & 0 & & & &
        \end{pNiceArray}
    }$}, \\
    &\M_{4} \coloneqq \mbox{\footnotesize $\expar{
        \begin{pNiceArray}{ccc:ccc}[margin]
            \Block{3-3}<\normalsize \boldmath>{z} & & & -xy & x^2 & xy^2 \\
            & & & y^2 & -xy & x^2 \\
            & & & x & y^2 & -xy \\
            \hdottedline
            0 & -xy & -x^2 & \Block{3-3}<\normalsize \boldmath>{z} & & \\
            -x & 0 & -xy & & & \\
            -y & -x & 0 & & &
        \end{pNiceArray},
        \begin{pNiceArray}{ccc:ccc}[margin]
            \Block{3-3}<\normalsize \boldmath>{z} & & & xy & -x^2 & -xy^2 \\
            & & & -y^2 & xy & -x^2 \\
            & & & -x & -y^2 & xy \\
            \hdottedline
            0 & xy & x^2 & \Block{3-3}<\normalsize \boldmath>{z} & & \\
            x & 0 & xy & & & \\
            y & x & 0 & & &
        \end{pNiceArray}
    }$}, \\
    &\M_{5} \coloneqq \mbox{\footnotesize $\expar{
        \begin{pNiceArray}{cc:cc}[margin]
            \Block{2-2}<\small \boldmath>{z} & & xy & x^2 \\
            & & x & -y^2 \\
            \hdottedline
            -y^2 & -x^2 & \Block{2-2}<\small \boldmath>{z} & \\
            -x & xy & &
        \end{pNiceArray},
        \begin{pNiceArray}{cc:cc}[margin]
            \Block{2-2}<\small \boldmath>{z} & & -xy & -x^2 \\
            & & -x & y^2 \\
            \hdottedline
            y^2 & x^2 & \Block{2-2}<\small \boldmath>{z} & \\
            x & -xy & &
        \end{pNiceArray}
    }$}, \\
    &\M_{6} \coloneqq \mbox{\footnotesize $\expar{
        \begin{pmatrix}
            z & -x^2 - y^3 \\
            x & z
        \end{pmatrix},
        \begin{pmatrix}
            z & x^2 + y^3 \\
            -x & z
        \end{pmatrix}
    }$}, \\
    &\M_{7} \coloneqq \mbox{\footnotesize $\expar{
        \begin{pNiceArray}{cc:cc}[margin]
            \Block{2-2}<\small \boldmath>{z} & & x^2 & xy^2 \\
            & & xy & -x^2 \\
            \hdottedline
            -x & -y^2 & \Block{2-2}<\small \boldmath>{z} & \\
            -y & x & &
        \end{pNiceArray},
        \begin{pNiceArray}{cc:cc}[margin]
            \Block{2-2}<\small \boldmath>{z} & & -x^2 & -xy^2 \\
            & & -xy & x^2 \\
            \hdottedline
            x & y^2 & \Block{2-2}<\small \boldmath>{z} & \\
            y & -x & &
        \end{pNiceArray}
    }$}.
\end{align*}

\subsection{\texorpdfstring{$E_{7}^1$ in characteristic $3$: $f = z^2 + x^3 + xy^3 + x^2y^2$}{E71InChar3}}
\label{subsec:E_7^1InChar3}
\begin{align*}
    &\M_{1} \coloneqq \mbox{\footnotesize $\expar{
        \begin{pNiceArray}{cc:cc}[margin]
            \Block{2-2}<\small \boldmath>{z} & & x^2 & xy^2 \\
            & & x + y & -x \\
            \hdottedline
            -x & -xy^2 & \Block{2-2}<\small \boldmath>{z} & \\
            -x - y & x^2 & &
        \end{pNiceArray},
        \begin{pNiceArray}{cc:cc}[margin]
            \Block{2-2}<\small \boldmath>{z} & & -x^2 & -xy^2 \\
            & & -x - y & x \\
            \hdottedline
            x & xy^2 & \Block{2-2}<\small \boldmath>{z} & \\
            x + y & -x^2 & &
        \end{pNiceArray}
    }$}, \\
    &\M_{2} \coloneqq \mbox{\footnotesize $\mleft\lparen
        \begin{pNiceArray}{ccc:ccc}[margin]
            \Block{3-3}<\normalsize \boldmath>{z} & & & x^2 & x^2y + xy^2 & 0 \\
            & & & -xy & x^2 & 0 \\
            & & & -y^2 & xy & x^2 + xy^2 + y^3 \\
            \hdottedline
            -x & xy + y^2 & 0 & \Block{3-3}<\normalsize \boldmath>{z} & & \\
            -y & -x & 0 & & & \\
            0 & y & -x & & &
        \end{pNiceArray} \mright.$,} \\
        \tag*{\footnotesize $
            \mleft. \begin{pNiceArray}{ccc:ccc}[margin]
                \Block{3-3}<\normalsize \boldmath>{z} & & & -x^2 & -x^2y - xy^2 & 0 \\
                & & & xy & -x^2 & 0 \\
                & & & y^2 & -xy & -x^2 - xy^2 - y^3 \\
                \hdottedline
                x & -xy - y^2 & 0 & \Block{3-3}<\normalsize \boldmath>{z} & & \\
                y & x & 0 & & & \\
                0 & -y & x & & &
            \end{pNiceArray} \mright\rparen
        $,} \\
    &\M_{3} \coloneqq \mbox{\footnotesize $\mleft\lparen
        \begin{pNiceArray}{cccc:cccc}[margin] 
            \Block{4-4}<\large \boldmath>{z} & & & & 0 & 0 & x^2 & x^2y + xy^2 \\
            & & & & 0 & 0 & -xy & x^2 \\
            & & & & -y^2 & x + y^2 & -x & -xy \\
            & & & & -x & -y & y & 0 \\
            \hdottedline
            0 & y^2 & xy & x^2 + xy^2 & \Block{4-4}<\large \boldmath>{z} & & & \\
            -x & y^2 & -x^2 & xy^2 & & & & \\
            -x & xy + y^2 & 0 & 0 & & & & \\
            -y & -x & 0 & 0 & & & &
        \end{pNiceArray} \mright.$,} \\
            \tag*{\footnotesize $
                \mleft. \begin{pNiceArray}{cccc:cccc}[margin]
                    \Block{4-4}<\large \boldmath>{z} & & & & 0 & 0 & -x^2 & -x^2y - xy^2 \\
                    & & & & 0 & 0 & xy & -x^2 \\
                    & & & & y^2 & -x - y^2 & x & xy \\
                    & & & & x & y & -y & 0 \\
                    \hdottedline
                    0 & -y^2 & -xy & -x^2 - xy^2 & \Block{4-4}<\large \boldmath>{z} & & & \\
                    x & -y^2 & x^2 & -xy^2 & & & & \\
                    x & -xy - y^2 & 0 & 0 & & & & \\
                    y & x & 0 & 0 & & & &
                \end{pNiceArray} \mright\rparen
            $,} \\
    &\M_{4} \coloneqq \mbox{\footnotesize $\mleft\lparen
        \begin{pNiceArray}{ccc:ccc}[margin]
            \Block{3-3}<\normalsize \boldmath>{z} & & & 0 & x^2 + xy^2 + y^3 & 0 \\
            & & & y^2 & -xy - y^3 & x^2 + xy^2 \\
            & & & x & y^2 & -xy \\
            \hdottedline
            0 & -xy & -xy^2 - x^2 & \Block{3-3}<\normalsize \boldmath>{z} & & \\
            -x & 0 & 0 & & & \\
            -y & -x & y^2 & & &
        \end{pNiceArray} \mright.$,} \\
        \tag*{\footnotesize $
            \mleft. \begin{pNiceArray}{ccc:ccc}[margin]
                \Block{3-3}<\normalsize \boldmath>{z} & & & 0 & -x^2 - xy^2 - y^3 & 0 \\
                & & & -y^2 & xy + y^3 & -x^2 - xy^2 \\
                & & & -x & -y^2 & xy \\
                \hdottedline
                0 & xy & xy^2 + x^2 & \Block{3-3}<\normalsize \boldmath>{z} & & \\
                x & 0 & 0 & & & \\
                y & x & -y^2 & & &
            \end{pNiceArray} \mright\rparen
        $,} \\
    &\M_{5} \coloneqq \mbox{\footnotesize $\expar{
        \begin{pNiceArray}{cc:cc}[margin]
            \Block{2-2}<\small \boldmath>{z} & & x^2 + xy & x^2 \\
            & & x & -y^2 \\
            \hdottedline
            -y^2 & -x^2 & \Block{2-2}<\small \boldmath>{z} & \\
            -x & x^2 + xy & &
        \end{pNiceArray},
        \begin{pNiceArray}{cc:cc}[margin]
            \Block{2-2}<\small \boldmath>{z} & & -x^2 - xy & -x^2 \\
            & & -x & y^2 \\
            \hdottedline
            y^2 & x^2 & \Block{2-2}<\small \boldmath>{z} & \\
            x & -x^2 - xy & &
        \end{pNiceArray}
    }$}, \\
    &\M_{6} \coloneqq \mbox{\footnotesize $\expar{
        \begin{pmatrix}
            z & -x^2 - xy^2 - y^3 \\
            x & z
        \end{pmatrix},
        \begin{pmatrix}
            z & x^2 + xy^2 + y^3 \\
            -x & z
        \end{pmatrix}
    }$}, \\
    &\M_{7} \coloneqq \mbox{\footnotesize $\expar{
        \begin{pNiceArray}{cc:cc}[margin]
            \Block{2-2}<\small \boldmath>{z} & & x^2 & xy^2 \\
            & & x^2 + xy & -x^2 \\
            \hdottedline
            -x & -y^2 & \Block{2-2}<\small \boldmath>{z} & \\
            -x - y & x & &
        \end{pNiceArray},
        \begin{pNiceArray}{cc:cc}[margin]
            \Block{2-2}<\small \boldmath>{z} & & -x^2 & -xy^2 \\
            & & -x^2 - xy & x^2 \\
            \hdottedline
            x & y^2 & \Block{2-2}<\small \boldmath>{z} & \\
            x + y & -x & &
        \end{pNiceArray}
    }$}.
\end{align*}

\subsection{\texorpdfstring{$E_{7}^r$ in characteristic $2$: $f = z^2 + x^3 + xy^3 + zg$}{E7rInChar2}}
\label{subsec:E_7^rInChar2}
Here, 
\[ 
    g \coloneqq
    \begin{cases}
        0 & \text{if $r = 0$,} \\
        x^2y & \text{if $r = 1$,} \\
        y^3 & \text{if $r = 2$,} \\
        xy & \text{if $r = 3$.}
    \end{cases}
\]

\begin{align*}
    &\M_{1} \coloneqq \mbox{\footnotesize $\expar{
        \begin{pNiceArray}{cc:w{c}{2em}w{c}{1em}}[margin]
            \Block{2-2}<\small \boldmath>{z} & & x^2 & xy^2 \\
            & & y & x \\
            \hdottedline
            x & xy^2 & \Block{2-2}<\small \boldmath>{z + g} & \\
            y & x^2 & &
        \end{pNiceArray},
        \begin{pNiceArray}{w{c}{1em}w{c}{2em}:cc}[margin]
            \Block{2-2}<\small \boldmath>{z + g} & & x^2 & xy^2 \\
            & & y & x \\
            \hdottedline
            x & xy^2 & \Block{2-2}<\small \boldmath>{z} & \\
            y & x^2 & &
        \end{pNiceArray}
    }$}, \\
    &\M_{2} \coloneqq \mbox{\footnotesize $\expar{
        \begin{pNiceArray}{ccc:ccc}[margin]
            \Block{3-3}<\normalsize \boldmath>{z} & & & x^2 & xy^2 & x^2y \\
            & & & xy & x^2 & xy^2 \\
            & & & y^2 & xy & x^2 \\
            \hdottedline
            x & 0 & xy & \Block{3-3}<\normalsize \boldmath>{z + g} & & \\
            y & x & 0 & & & \\
            0 & y & x & & &
        \end{pNiceArray},
        \begin{pNiceArray}{ccc:ccc}[margin]
            \Block{3-3}<\normalsize \boldmath>{z + g} & & & x^2 & xy^2 & x^2y \\
            & & & xy & x^2 & xy^2 \\
            & & & y^2 & xy & x^2 \\
            \hdottedline
            x & 0 & xy & \Block{3-3}<\normalsize \boldmath>{z} & & \\
            y & x & 0 & & & \\
            0 & y & x & & &
        \end{pNiceArray}
    }$}, \\
    &\M_{3} \coloneqq \mbox{\footnotesize $\expar{
        \begin{pNiceArray}{cccc:cccc}[margin] 
            \Block{4-4}<\large \boldmath>{z} & & & & 0 & 0 & x^2 & xy^2 \\
            & & & & 0 & 0 & xy & x^2 \\
            & & & & x & y^2 & 0 & xy \\
            & & & & y & x & x & 0 \\
            \hdottedline
            0 & xy & x^2 & xy^2 & \Block{4-4}<\large \boldmath>{z + g} & & & \\
            x & 0 & xy & x^2 & & & & \\
            x & y^2 & 0 & 0 & & & & \\
            y & x & 0 & 0 & & & &
        \end{pNiceArray},
        \begin{pNiceArray}{cccc:cccc}[margin]
            \Block{4-4}<\large \boldmath>{z + g} & & & & 0 & 0 & x^2 & xy^2 \\
            & & & & 0 & 0 & xy & x^2 \\
            & & & & x & y^2 & 0 & xy \\
            & & & & y & x & x & 0 \\
            \hdottedline
            0 & xy & x^2 & xy^2 & \Block{4-4}<\large \boldmath>{z} & & & \\
            x & 0 & xy & x^2 & & & & \\
            x & y^2 & 0 & 0 & & & & \\
            y & x & 0 & 0 & & & &
        \end{pNiceArray}
    }$}, \\
    &\M_{4} \coloneqq \mbox{\footnotesize $\expar{
        \begin{pNiceArray}{ccc:ccc}[margin]
            \Block{3-3}<\normalsize \boldmath>{z} & & & xy & x^2 & xy^2 \\
            & & & y^2 & xy & x^2 \\
            & & & x & y^2 & xy \\
            \hdottedline
            0 & xy & x^2 & \Block{3-3}<\normalsize \boldmath>{z + g} & & \\
            x & 0 & xy & & & \\
            y & x & 0 & & &
        \end{pNiceArray},
        \begin{pNiceArray}{ccc:ccc}[margin]
            \Block{3-3}<\normalsize \boldmath>{z + g} & & & xy & x^2 & xy^2 \\
            & & & y^2 & xy & x^2 \\
            & & & x & y^2 & xy \\
            \hdottedline
            0 & xy & x^2 & \Block{3-3}<\normalsize \boldmath>{z} & & \\
            x & 0 & xy & & & \\
            y & x & 0 & & &
        \end{pNiceArray}
    }$}, \\
    &\M_{5} \coloneqq \mbox{\footnotesize $\expar{
        \begin{pNiceArray}{cc:w{c}{2em}w{c}{1em}}[margin]
            \Block{2-2}<\small \boldmath>{z} & & xy & x^2 \\
            & & x & y^2 \\
            \hdottedline
            y^2 & x^2 & \Block{2-2}<\small \boldmath>{z + g} & \\
            x & xy & &
        \end{pNiceArray},
        \begin{pNiceArray}{w{c}{1em}w{c}{2em}:cc}[margin]
            \Block{2-2}<\small \boldmath>{z + g} & & xy & x^2 \\
            & & x & y^2 \\
            \hdottedline
            y^2 & x^2 & \Block{2-2}<\small \boldmath>{z} & \\
            x & xy & &
        \end{pNiceArray}
    }$}, \\
    &\M_{6} \coloneqq \mbox{\footnotesize $\expar{
        \begin{pmatrix}
            z & x^2 + y^3 \\
            x & z + g
        \end{pmatrix},
        \begin{pmatrix}
            z + g & x^2 + y^3 \\
            x & z
        \end{pmatrix}
    }$}, \\
    &\M_{7} \coloneqq \mbox{\footnotesize $\expar{
        \begin{pNiceArray}{cc:w{c}{2em}w{c}{1em}}[margin]
            \Block{2-2}<\small \boldmath>{z} & & x^2 & xy^2 \\
            & & xy & x^2 \\
            \hdottedline
            x & y^2 & \Block{2-2}<\small \boldmath>{z + g} & \\
            y & x & &
        \end{pNiceArray},
        \begin{pNiceArray}{w{c}{1em}w{c}{2em}:cc}[margin]
            \Block{2-2}<\small \boldmath>{z + g} & & x^2 & xy^2 \\
            & & xy & x^2 \\
            \hdottedline
            x & y^2 & \Block{2-2}<\small \boldmath>{z} & \\
            y & x & &
        \end{pNiceArray}
    }$}.
\end{align*}

\subsection{\texorpdfstring{$E_{8}^0$: $f = z^2 + x^3 + y^5$}{E80}}
\label{subsec:E_8^0}
\begin{align*}
    &\M_{1} \coloneqq \mbox{\footnotesize $\expar{
        \begin{pNiceArray}{cc:cc}[margin]
            \Block{2-2}<\small \boldmath>{z} & & -y^3 & -x^2 \\
            & & x & -y^2 \\
            \hdottedline
            y^2 & -x^2 & \Block{2-2}<\small \boldmath>{z} & \\
            x & y^3 & &
        \end{pNiceArray},
        \begin{pNiceArray}{cc:cc}[margin]
            \Block{2-2}<\small \boldmath>{z} & & y^3 & x^2 \\
            & & -x & y^2 \\
            \hdottedline
            -y^2 & x^2 & \Block{2-2}<\small \boldmath>{z} & \\
            -x & -y^3 & &
        \end{pNiceArray}
    }$}, \\
    &\M_{2} \coloneqq \mbox{\footnotesize $\expar{
        \begin{pNiceArray}{cccc:cccc}[margin] 
            \Block{4-4}<\large \boldmath>{z} & & & & 0 & -y^3 & -x^2 & 0 \\
            & & & & -y^2 & 0 & xy & -x^2 \\
            & & & & -x & -y^2 & 0 & y^3 \\
            & & & & 0 & -x & y^2 & 0 \\
            \hdottedline
            0 & y^3 & x^2 & -xy^2 & \Block{4-4}<\large \boldmath>{z} & & & \\
            y^2 & 0 & 0 & x^2 & & & & \\
            x & 0 & 0 & -y^3 & & & & \\
            y & x & -y^2 & 0 & & & &
        \end{pNiceArray},
        \begin{pNiceArray}{cccc:cccc}[margin]
            \Block{4-4}<\large \boldmath>{z} & & & & 0 & y^3 & x^2 & 0 \\
            & & & & y^2 & 0 & -xy & x^2 \\
            & & & & x & y^2 & 0 & -y^3 \\
            & & & & 0 & x & -y^2 & 0 \\
            \hdottedline
            0 & -y^3 & -x^2 & xy^2 & \Block{4-4}<\large \boldmath>{z} & & & \\
            -y^2 & 0 & 0 & -x^2 & & & & \\
            -x & 0 & 0 & y^3 & & & & \\
            -y & -x & y^2 & 0 & & & &
        \end{pNiceArray}
    }$}, \\
    &\M_{3} \coloneqq \mbox{\footnotesize $\mleft\lparen
        \begin{pNiceArray}{cccccc:cccccc}[margin] 
            \Block{6-6}<\LARGE \boldmath>{z} & & & & & & 0 & 0 & 0 & -x^2 & xy^2 & -y^4 \\
            & & & & & & 0 & 0 & 0 & -y^3 & -x^2 & xy^2 \\
            & & & & & & 0 & 0 & 0 & xy & -y^3 & -x^2 \\
            & & & & & & -x & -y^2 & 0 & 0 & 0 & y^3 \\
            & & & & & & 0 & -x & -y^2 & y^2 & 0 & 0 \\
            & & & & & & -y & 0 & -x & 0 & y^2 & 0 \\
            \hdottedline
            0 & 0 & y^3 & x^2 & -xy^2 & y^4 & \Block{6-6}<\LARGE \boldmath>{z} & & & & & \\
            y^2 & 0 & 0 & y^3 & x^2 & -xy^2 & & & & & & \\
            0 & y^2 & 0 & -xy & y^3 & x^2 & & & & & & \\
            x & y^2 & 0 & 0 & 0 & 0 & & & & & & \\
            0 & x & y^2 & 0 & 0 & 0 & & & & & & \\
            y & 0 & x & 0 & 0 & 0 & & & & & &
        \end{pNiceArray} \mright.$,} \\
            \tag*{\footnotesize $
                \mleft. \begin{pNiceArray}{cccccc:cccccc}[margin]
                    \Block{6-6}<\LARGE \boldmath>{z} & & & & & & 0 & 0 & 0 & x^2 & -xy^2 & y^4 \\
                    & & & & & & 0 & 0 & 0 & y^3 & x^2 & -xy^2 \\
                    & & & & & & 0 & 0 & 0 & -xy & y^3 & x^2 \\
                    & & & & & & x & y^2 & 0 & 0 & 0 & -y^3 \\
                    & & & & & & 0 & x & y^2 & -y^2 & 0 & 0 \\
                    & & & & & & y & 0 & x & 0 & -y^2 & 0 \\
                    \hdottedline
                    0 & 0 & -y^3 & -x^2 & xy^2 & -y^4 & \Block{6-6}<\LARGE \boldmath>{z} & & & & & \\
                    -y^2 & 0 & 0 & -y^3 & -x^2 & xy^2 & & & & & & \\
                    0 & -y^2 & 0 & xy & -y^3 & -x^2 & & & & & & \\
                    -x & -y^2 & 0 & 0 & 0 & 0 & & & & & & \\
                    0 & -x & -y^2 & 0 & 0 & 0 & & & & & & \\
                    -y & 0 & -x & 0 & 0 & 0 & & & & & &
                \end{pNiceArray} \mright\rparen
            $,} \\
    &\M_{4} \coloneqq \mbox{\footnotesize $\mleft\lparen
        \begin{pNiceArray}{ccccc:ccccc}[margin] 
            \Block{5-5}<\Large \boldmath>{z} & & & & & -y^3 & x^2 & 0 & 0 & 0 \\
            & & & & & 0 & y^3 & -x^2 & xy^2 & -y^4 \\
            & & & & & 0 & -xy & -y^3 & -x^2 & xy^2 \\
            & & & & & y^2 & 0 & xy & -y^3 & -x^2 \\
            & & & & & -x & -y^2 & 0 & 0 & 0 \\
            \hdottedline
            y^2 & 0 & 0 & 0 & x^2 & \Block{5-5}<\Large \boldmath>{z} & & & & \\
            -x & 0 & 0 & 0 & y^3 & & & & & \\
            0 & x & y^2 & 0 & 0 & & & & & \\
            y & 0 & x & y^2 & 0 & & & & & \\
            0 & y & 0 & x & y^2 & & & & &
        \end{pNiceArray} \mright.$,} \\
            \tag*{\footnotesize $
                \mleft. \begin{pNiceArray}{ccccc:ccccc}[margin]
                    \Block{5-5}<\Large \boldmath>{z} & & & & & y^3 & -x^2 & 0 & 0 & 0 \\
                    & & & & & 0 & -y^3 & x^2 & -xy^2 & y^4 \\
                    & & & & & 0 & xy & y^3 & x^2 & -xy^2 \\
                    & & & & & -y^2 & 0 & -xy & y^3 & x^2 \\
                    & & & & & x & y^2 & 0 & 0 & 0 \\
                    \hdottedline
                    -y^2 & 0 & 0 & 0 & -x^2 & \Block{5-5}<\Large \boldmath>{z} & & & & \\
                    x & 0 & 0 & 0 & -y^3 & & & & & \\
                    0 & -x & -y^2 & 0 & 0 & & & & & \\
                    -y & 0 & -x & -y^2 & 0 & & & & & \\
                    0 & -y & 0 & -x & -y^2 & & & & &
                \end{pNiceArray} \mright\rparen
            $,} \\
    &\M_{5} \coloneqq \mbox{\footnotesize $\expar{
        \begin{pNiceArray}{cccc:cccc}[margin] 
            \Block{4-4}<\large \boldmath>{z} & & & & xy & -y^2 & -x^2 & 0 \\
            & & & & -y^3 & 0 & 0 & -x \\
            & & & & x^2 & 0 & 0 & -y^2 \\
            & & & & 0 & x & -y^3 & -y \\
            \hdottedline
            0 & y^2 & -x & 0 & \Block{4-4}<\large \boldmath>{z} & & & \\
            y^3 & xy & 0 & -x^2 & & & & \\
            x & 0 & -y & y^2 & & & & \\
            0 & x^2 & y^3 & 0 & & & &
        \end{pNiceArray},
        \begin{pNiceArray}{cccc:cccc}[margin]
            \Block{4-4}<\large \boldmath>{z} & & & & -xy & y^2 & x^2 & 0 \\
            & & & & y^3 & 0 & 0 & x \\
            & & & & -x^2 & 0 & 0 & y^2 \\
            & & & & 0 & -x & y^3 & y \\
            \hdottedline
            0 & -y^2 & x & 0 & \Block{4-4}<\large \boldmath>{z} & & & \\
            -y^3 & -xy & 0 & x^2 & & & & \\
            -x & 0 & y & -y^2 & & & & \\
            0 & -x^2 & -y^3 & 0 & & & &
        \end{pNiceArray}
    }$}, \\
    &\M_{6} \coloneqq \mbox{\footnotesize $\expar{
        \begin{pNiceArray}{ccc:ccc}[margin]
            \Block{3-3}<\normalsize \boldmath>{z} & & & -x^2 & -y^4 & xy^3 \\
            & & & xy & -x^2 & -y^4 \\
            & & & -y^2 & xy & -x^2 \\
            \hdottedline
            x & 0 & y^3 & \Block{3-3}<\normalsize \boldmath>{z} & & \\
            y & x & 0 & & & \\
            0 & y & x & & &
        \end{pNiceArray},
        \begin{pNiceArray}{ccc:ccc}[margin]
            \Block{3-3}<\normalsize \boldmath>{z} & & & x^2 & y^4 & -xy^3 \\
            & & & -xy & x^2 & y^4 \\
            & & & y^2 & -xy & x^2 \\
            \hdottedline
            -x & 0 & -y^3 & \Block{3-3}<\normalsize \boldmath>{z} & & \\
            -y & -x & 0 & & & \\
            0 & -y & -x & & &
        \end{pNiceArray}
    }$}, \\
    &\M_{7} \coloneqq \mbox{\footnotesize $\expar{
        \begin{pNiceArray}{cc:cc}[margin]
            \Block{2-2}<\small \boldmath>{z} & & -x^2 & -y^4 \\
            & & -y & x \\
            \hdottedline
            x & y^4 & \Block{2-2}<\small \boldmath>{z} & \\
            y & -x^2 & &
        \end{pNiceArray},
        \begin{pNiceArray}{cc:cc}[margin]
            \Block{2-2}<\small \boldmath>{z} & & x^2 & y^4 \\
            & & y & -x \\
            \hdottedline
            -x & -y^4 & \Block{2-2}<\small \boldmath>{z} & \\
            -y & x^2 & &
        \end{pNiceArray}
    }$}, \\
    &\M_{8} \coloneqq \mbox{\footnotesize $\expar{
        \begin{pNiceArray}{ccc:ccc}[margin]
            \Block{3-3}<\normalsize \boldmath>{z} & & & -x^2 & xy^2 & -y^4 \\
            & & & -y^3 & -x^2 & xy^2 \\
            & & & xy & -y^3 & -x^2 \\
            \hdottedline
            x & y^2 & 0 & \Block{3-3}<\normalsize \boldmath>{z} & & \\
            0 & x & y^2 & & & \\
            y & 0 & x & & &
        \end{pNiceArray},
        \begin{pNiceArray}{ccc:ccc}[margin]
            \Block{3-3}<\normalsize \boldmath>{z} & & & x^2 & -xy^2 & y^4 \\
            & & & y^3 & x^2 & -xy^2 \\
            & & & -xy & y^3 & x^2 \\
            \hdottedline
            -x & -y^2 & 0 & \Block{3-3}<\normalsize \boldmath>{z} & & \\
            0 & -x & -y^2 & & & \\
            -y & 0 & -x & & &
        \end{pNiceArray}
    }$}. 
\end{align*}

\subsection{\texorpdfstring{$E_{8}^1$ in characteristic $5$: $f = z^2 + x^3 + y^5 + xy^4$}{E81InChar5}}
\label{subsec:E_8^1InChar5}
\begin{align*}
    &\M_{1} \coloneqq \mbox{\footnotesize $\expar{
        \begin{pNiceArray}{cc:cc}[margin]
            \Block{2-2}<\small  \boldmath>{z} & & -y^3 & -x^2 - y^4 \\
            & & x & -y^2 \\
            \hdottedline
            y^2 & -x^2 - y^4 & \Block{2-2}<\small  \boldmath>{z} & \\
            x & y^3 & &
        \end{pNiceArray},
        \begin{pNiceArray}{cc:cc}[margin]
            \Block{2-2}<\small  \boldmath>{z} & & y^3 & x^2 + y^4 \\
            & & -x & y^2 \\
            \hdottedline
            -y^2 & x^2 + y^4 & \Block{2-2}<\small  \boldmath>{z} & \\
            -x & -y^3 & &
        \end{pNiceArray}
    }$}, \\
    &\M_{2} \coloneqq \mbox{\footnotesize $\mleft\lparen
        \begin{pNiceArray}{cccc:cccc}[margin] 
            \Block{4-4}<\large  \boldmath>{z} & & & & 0 & -y^3 & -x^2 - y^4 & 0 \\
            & & & & -y^2 & y^3 & xy & -x^2 - y^4 \\
            & & & & -x & -y^2 & 0 & y^3 \\
            & & & & 0 & -x & y^2 & 0 \\
            \hdottedline
            0 & y^3 & x^2 + y^4 & -xy^2 & \Block{4-4}<\large  \boldmath>{z} & & & \\
            y^2 & 0 & 0 & x^2 + y^4 & & & & \\
            x & 0 & 0 & -y^3 & & & & \\
            y & x & -y^2 & y^3 & & & &
        \end{pNiceArray} \mright.$,} \\
            \tag*{\footnotesize $
                \mleft. \begin{pNiceArray}{cccc:cccc}[margin]
                    \Block{4-4}<\large  \boldmath>{z} & & & & 0 & y^3 & x^2 + y^4 & 0 \\
                    & & & & y^2 & -y^3 & -xy & x^2 + y^4 \\
                    & & & & x & y^2 & 0 & -y^3 \\
                    & & & & 0 & x & -y^2 & 0 \\
                    \hdottedline
                    0 & -y^3 & -x^2 - y^4 & xy^2 & \Block{4-4}<\large  \boldmath>{z} & & & \\
                    -y^2 & 0 & 0 & -x^2 - y^4 & & & & \\
                    -x & 0 & 0 & y^3 & & & & \\
                    -y & -x & y^2 & -y^3 & & & &
                \end{pNiceArray} \mright\rparen
            $,} \\
    &\M_{3} \coloneqq \mbox{\footnotesize $\mleft\lparen
        \begin{pNiceArray}{cccccc:cccccc}[margin] 
            \Block{6-6}<\LARGE  \boldmath>{z} & & & & & & 0 & y^3 & 0 & -x^2 - y^4 & 0 & 0 \\
            & & & & & & -y^2 & y^3 & 0 & -xy & 0 & x^2 + y^4 \\
            & & & & & & y & -y^2 & -xy - y^2 & 0 & -x & -y^3 \\
            & & & & & & 0 & -x & 0 & -y^2 & 0 & 0 \\
            & & & & & & -x & -y^2 & 0 & 0 & 0 & -y^3 \\
            & & & & & & 0 & -xy & -x^2 & 0 & y^3 & -xy^2 \\
            \hdottedline
            0 & y^3 & 0 & -xy^2 & x^2 + y^4 & 0 & \Block{6-6}<\LARGE  \boldmath>{z} & & & & & \\
            -y^2 & 0 & 0 & x^2 + y^4 & 0 & 0 & & & & & & \\
            0 & y^2 & y^3 & -xy & 0 & x & & & & & & \\
            x & 0 & 0 & y^3 & 0 & 0 & & & & & & \\
            0 & 0 & x^2 & 0 & xy & -xy - y^2 & & & & & & \\
            y & -x & 0 & -y^3 & y^2 & 0 & & & & & &
        \end{pNiceArray} \mright.$,} \\
            \tag*{\footnotesize $
                \mleft. \begin{pNiceArray}{cccccc:cccccc}[margin]
                    \Block{6-6}<\LARGE  \boldmath>{z} & & & & & & 0 & -y^3 & 0 & x^2 + y^4 & 0 & 0 \\
                    & & & & & & y^2 & -y^3 & 0 & xy & 0 & -x^2 - y^4 \\
                    & & & & & & -y & y^2 & xy + y^2 & 0 & x & y^3 \\
                    & & & & & & 0 & x & 0 & y^2 & 0 & 0 \\
                    & & & & & & x & y^2 & 0 & 0 & 0 & y^3 \\
                    & & & & & & 0 & xy & x^2 & 0 & -y^3 & xy^2 \\
                    \hdottedline
                    0 & -y^3 & 0 & xy^2 & -x^2 - y^4 & 0 & \Block{6-6}<\LARGE  \boldmath>{z} & & & & & \\
                    y^2 & 0 & 0 & -x^2 - y^4 & 0 & 0 & & & & & & \\
                    0 & -y^2 & -y^3 & xy & 0 & -x & & & & & & \\
                    -x & 0 & 0 & -y^3 & 0 & 0 & & & & & & \\
                    0 & 0 & -x^2 & 0 & -xy & xy + y^2 & & & & & & \\
                    -y & x & 0 & y^3 & -y^2 & 0 & & & & & &
                \end{pNiceArray} \mright\rparen
            $,} \\
    &\M_{4} \coloneqq \mbox{\footnotesize $\mleft\lparen
        \begin{pNiceArray}{ccccc:ccccc}[margin] 
            \Block{5-5}<\Large \boldmath>{z} & & & & & -y^3 & x^2 & 0 & -xy^3 & 0 \\
            & & & & & 0 & y^3 & x^2 + y^4 & xy^2 & -y^4 \\
            & & & & & 0 & -xy & y^3 & -x^2 & xy^2 \\
            & & & & & y^2 & 0 & -xy & -y^3 & -x^2 \\
            & & & & & -x & -y^2 & -y^3 & 0 & -xy^2 \\
            \hdottedline
            y^2 & 0 & 0 & -xy^2 & x^2 & \Block{5-5}<\Large \boldmath>{z} & & & & \\
            -x & 0 & y^3 & 0 & y^3 & & & & & \\
            0 & -x & -y^2 & 0 & 0 & & & & & \\
            y & 0 & x & y^2 & 0 & & & & & \\
            0 & y & 0 & x & y^2 & & & & &
        \end{pNiceArray} \mright.$,} \\
            \tag*{\footnotesize $
                \mleft. \begin{pNiceArray}{ccccc:ccccc}[margin]
                    \Block{5-5}<\Large \boldmath>{z} & & & & & y^3 & -x^2 & 0 & xy^3 & 0 \\
                    & & & & & 0 & -y^3 & -x^2 - y^4 & -xy^2 & y^4 \\
                    & & & & & 0 & xy & -y^3 & x^2 & -xy^2 \\
                    & & & & & -y^2 & 0 & xy & y^3 & x^2 \\
                    & & & & & x & y^2 & y^3 & 0 & xy^2 \\
                    \hdottedline
                    -y^2 & 0 & 0 & xy^2 & -x^2 & \Block{5-5}<\Large \boldmath>{z} & & & & \\
                    x & 0 & -y^3 & 0 & -y^3 & & & & & \\
                    0 & x & y^2 & 0 & 0 & & & & & \\
                    -y & 0 & -x & -y^2 & 0 & & & & & \\
                    0 & -y & 0 & -x & -y^2 & & & & &
                \end{pNiceArray} \mright\rparen
            $,} \\
    &\M_{5} \coloneqq \mbox{\footnotesize $\mleft\lparen
        \begin{pNiceArray}{cccc:cccc}[margin] 
            \Block{4-4}<\large \boldmath>{z} & & & & xy & -xy - y^2 & -x^2 & 0 \\
            & & & & -y^3 & 0 & 0 & -x \\
            & & & & x^2 & 0 & 0 & -xy - y^2 \\
            & & & & 0 & x & -y^3 & -y \\
            \hdottedline
            0 & xy + y^2 & -x & 0 & \Block{4-4}<\large \boldmath>{z} & & & \\
            y^3 & xy & 0 & -x^2 & & & & \\
            x & 0 & -y & xy + y^2 & & & & \\
            0 & x^2 & y^3 & 0 & & & &
        \end{pNiceArray} \mright.$,} \\
            \tag*{\footnotesize $
                \mleft. \begin{pNiceArray}{cccc:cccc}[margin]
                    \Block{4-4}<\large \boldmath>{z} & & & & -xy & xy + y^2 & x^2 & 0 \\
                    & & & & y^3 & 0 & 0 & x \\
                    & & & & -x^2 & 0 & 0 & xy + y^2 \\
                    & & & & 0 & -x & y^3 & y \\
                    \hdottedline
                    0 & -xy - y^2 & x & 0 & \Block{4-4}<\large \boldmath>{z} & & & \\
                    -y^3 & -xy & 0 & x^2 & & & & \\
                    -x & 0 & y & -xy - y^2 & & & & \\
                    0 & -x^2 & -y^3 & 0 & & & &
                \end{pNiceArray} \mright\rparen
            $,} \\
    &\M_{6} \coloneqq \mbox{\footnotesize $\expar{
        \begin{pNiceArray}{ccc:ccc}[margin]
            \Block{3-3}<\normalsize \boldmath>{z} & & & -x^2 & -xy^3 - y^4 & xy^3 \\
            & & & xy & -x^2 & -y^4 \\
            & & & -y^2 & xy & -x^2 - y^4 \\
            \hdottedline
            x & -y^3 & y^3 & \Block{3-3}<\normalsize \boldmath>{z} & & \\
            y & x & 0 & & & \\
            0 & y & x & & &
        \end{pNiceArray},
        \begin{pNiceArray}{ccc:ccc}[margin]
            \Block{3-3}<\normalsize \boldmath>{z} & & & x^2 & xy^3 + y^4 & -xy^3 \\
            & & & -xy & x^2 & y^4 \\
            & & & y^2 & -xy & x^2 + y^4 \\
            \hdottedline
            -x & y^3 & -y^3 & \Block{3-3}<\normalsize \boldmath>{z} & & \\
            -y & -x & 0 & & & \\
            0 & -y & -x & & &
        \end{pNiceArray}
    }$}, \\
    &\M_{7} \coloneqq \mbox{\footnotesize $\expar{
        \begin{pNiceArray}{cc:cc}[margin]
            \Block{2-2}<\small \boldmath>{z} & & -x^2 & -xy^3 - y^4 \\
            & & -y & x \\
            \hdottedline
            x & xy^3 + y^4 & \Block{2-2}<\small \boldmath>{z} & \\
            y & -x^2 & &
        \end{pNiceArray},
        \begin{pNiceArray}{cc:cc}[margin]
            \Block{2-2}<\small \boldmath>{z} & & x^2 & xy^3 + y^4 \\
            & & y & -x \\
            \hdottedline
            -x & -xy^3 - y^4 & \Block{2-2}<\small \boldmath>{z} & \\
            -y & x^2 & &
        \end{pNiceArray}
    }$}, \\
    &\M_{8} \coloneqq \mbox{\footnotesize $\expar{
        \begin{pNiceArray}{ccc:ccc}[margin]
            \Block{3-3}<\normalsize \boldmath>{z} & & & -x^2 & xy^2 & -xy^3 - y^4 \\
            & & &  -y^3 & -x^2 - y^4 & xy^2 \\
            & & & xy & -y^3 & -x^2 \\
            \hdottedline
            x & y^2 & -y^3 & \Block{3-3}<\normalsize \boldmath>{z} & & \\
            0 & x & y^2 & & & \\
            y & 0 & x & & &
        \end{pNiceArray},
        \begin{pNiceArray}{ccc:ccc}[margin]
            \Block{3-3}<\normalsize \boldmath>{z} & & & x^2 & -xy^2 & xy^3 + y^4 \\
            & & &  y^3 & x^2 + y^4 & -xy^2 \\
            & & & -xy & y^3 & x^2 \\
            \hdottedline
            -x & -y^2 & y^3 & \Block{3-3}<\normalsize \boldmath>{z} & & \\
            0 & -x & -y^2 & & & \\
            -y & 0 & -x & & &
        \end{pNiceArray}
    }$}. \\
\end{align*}

\subsection{\texorpdfstring{$E_{8}^1$ in characteristic $3$: $f = z^2 + x^3 + y^5 + x^2y^3$}{E81InChar3}}
\label{subsec:E_8^1InChar3}
\begin{align*}
    &\M_{1} \coloneqq \mbox{\footnotesize $\expar{
        \begin{pNiceArray}{cc:cc}[margin]
            \Block{2-2}<\small  \boldmath>{z} & & - x^2 - xy^3 & y^2 \\
            & & -y^3 & -x \\
            \hdottedline
            x & y^2 & \Block{2-2}<\small  \boldmath>{z} & \\
            -y^3 & x^2 + xy^3 & &
        \end{pNiceArray},
        \begin{pNiceArray}{cc:cc}[margin]
            \Block{2-2}<\small  \boldmath>{z} & &  x^2 + xy^3 & -y^2 \\
            & & y^3 & x \\
            \hdottedline
            -x & -y^2 & \Block{2-2}<\small  \boldmath>{z} & \\
            y^3 & - x^2 - xy^3 & &
        \end{pNiceArray}
    }$}, \\
    &\M_{2} \coloneqq \mbox{\footnotesize $\mleft\lparen
        \begin{pNiceArray}{cccc:cccc}[margin] 
            \Block{4-4}<\large  \boldmath>{z} & & & & 0 & -y^3 & -x^2 - xy^3 & -y^4 \\
            & & & & y^2 & xy^2 & xy & -x^2 \\
            & & & & 0 & x & -y^2 & xy \\
            & & & & -x & y^2 & 0 & 0\\
            \hdottedline
            0 & -y^3 & -xy^2 & x^2 + xy^3 & \Block{4-4}<\large  \boldmath>{z} & & & \\
            0 & -xy & -x^2 & -y^3 & & & & \\
            x & 0 & y^3 & 0 & & & & \\
            y & x & -xy^2 & y^2 & & & &
        \end{pNiceArray} \mright.$,} \\
            \tag*{\footnotesize $
                \mleft. \begin{pNiceArray}{cccc:cccc}[margin]
                    \Block{4-4}<\large  \boldmath>{z} & & & & 0 & y^3 & x^2 + xy^3 & y^4 \\
                    & & & & -y^2 & -xy^2 & -xy & x^2 \\
                    & & & & 0 & -x & y^2 & -xy \\
                    & & & & x & -y^2 & 0 & 0\\
                    \hdottedline
                    0 & y^3 & xy^2 & -x^2 - xy^3 & \Block{4-4}<\large  \boldmath>{z} & & & \\
                    0 & xy & x^2 & y^3 & & & & \\
                    -x & 0 & -y^3 & 0 & & & & \\
                    -y & -x & xy^2 & -y^2 & & & &
                \end{pNiceArray} \mright\rparen
            $,} \\
    &\M_{3} \coloneqq \mbox{\footnotesize $\mleft\lparen
        \begin{pNiceArray}{cccccc:cccccc}[margin] 
            \Block{6-6}<\LARGE  \boldmath>{z} & & & & & & -xy & 0 & 0 & -x^2 & 0 & -y^4 \\
            & & & & & & 0 & -xy & y^3 & 0 & - x^2 - xy^3 & 0 \\
            & & & & & & 0 & 0 & -xy & -xy & -y^3 & x^2 \\
            & & & & & & -x & -y^2 & 0 & xy^2 & 0 & 0 \\
            & & & & & & 0 & -x & 0 & -y^2 & 0 & xy \\
            & & & & & & y & 0 & -x & 0 & -y^2 & -xy^2 \\
            \hdottedline
            xy^2 & 0 & y^3 & x^2 & -xy^2 & -y^4 & \Block{6-6}<\LARGE  \boldmath>{z} & & & & & \\
            0 & 0 & -xy & y^3 & x^2 + xy^3 & xy^2 & & & & & & \\
            0 & -y^2 & xy^2 & xy & 0 & x^2 & & & & & & \\
            x & 0 & 0 & -xy & y^3 & 0 & & & & & & \\
            0 & x & y^2 & 0 & -xy & 0 & & & & & & \\
            y & 0 & -x & 0 & 0 & xy & & & & & &
        \end{pNiceArray} \mright.$,} \\
            \tag*{\footnotesize $
                \mleft. \begin{pNiceArray}{cccccc:cccccc}[margin]
                    \Block{6-6}<\LARGE  \boldmath>{z} & & & & & & xy & 0 & 0 & x^2 & 0 & y^4 \\
                    & & & & & & 0 & xy & -y^3 & 0 & x^2 + xy^3 & 0 \\
                    & & & & & & 0 & 0 & xy & xy & y^3 & -x^2 \\
                    & & & & & & x & y^2 & 0 & -xy^2 & 0 & 0 \\
                    & & & & & & 0 & x & 0 & y^2 & 0 & -xy \\
                    & & & & & & -y & 0 & x & 0 & y^2 & xy^2 \\
                    \hdottedline
                    -xy^2 & 0 & -y^3 & -x^2 & xy^2 & y^4 & \Block{6-6}<\LARGE  \boldmath>{z} & & & & & \\
                    0 & 0 & xy & -y^3 & -x^2 - xy^3 & -xy^2 & & & & & & \\
                    0 & y^2 & -xy^2 & -xy & 0 & -x^2 & & & & & & \\
                    -x & 0 & 0 & xy & -y^3 & 0 & & & & & & \\
                    0 & -x & -y^2 & 0 & xy & 0 & & & & & & \\
                    -y & 0 & x & 0 & 0 & -xy & & & & & &
                \end{pNiceArray} \mright\rparen
            $,} \\
    &\M_{4} \coloneqq \mbox{\footnotesize $\mleft\lparen
        \begin{pNiceArray}{ccccc:ccccc}[margin] 
            \Block{5-5}<\Large \boldmath>{z} & & & & & -xy & x^2 & 0 & -x^2y^2 - y^4 & 0 \\
            & & & & & 0 & -xy^2 & -x^2 - xy^3 & -x^2y & -y^4 \\
            & & & & & y^2 & -xy & 0 & -x^2 & 0 \\
            & & & & & -xy^2 & -y^3 & -xy & -xy^2 & x^2 \\
            & & & & & -x & 0 & y^2 & 0 & -xy \\
            \hdottedline
            0 & 0 & -y^3 & xy & x^2 & \Block{5-5}<\Large \boldmath>{z} & & & & \\
            -x & 0 & xy^2 & y^2 & xy & & & & & \\
            0 & x & -xy & 0 & -y^3 & & & & & \\
            y & 0 & x & 0 & 0 & & & & & \\
            0 & y & 0 & -x & xy^2 & & & & &
        \end{pNiceArray} \mright.$,} \\
            \tag*{\footnotesize $
                \mleft. \begin{pNiceArray}{ccccc:ccccc}[margin]
                    \Block{5-5}<\Large \boldmath>{z} & & & & & xy & -x^2 & 0 & x^2y^2 + y^4 & 0 \\
                    & & & & & 0 & xy^2 & x^2 + xy^3 & x^2y & y^4 \\
                    & & & & & -y^2 & xy & 0 & x^2 & 0 \\
                    & & & & & xy^2 & y^3 & xy & xy^2 & -x^2 \\
                    & & & & & x & 0 & -y^2 & 0 & xy \\
                    \hdottedline
                    0 & 0 & y^3 & -xy & -x^2 & \Block{5-5}<\Large \boldmath>{z} & & & & \\
                    x & 0 & -xy^2 & -y^2 & -xy & & & & & \\
                    0 & -x & xy & 0 & y^3 & & & & & \\
                    -y & 0 & -x & 0 & 0 & & & & & \\
                    0 & -y & 0 & x & -xy^2 & & & & &
                \end{pNiceArray} \mright\rparen
            $,} \\
    &\M_{5} \coloneqq \mbox{\footnotesize $\mleft\lparen
        \begin{pNiceArray}{cccc:cccc}[margin] 
            \Block{4-4}<\large \boldmath>{z} & & & & xy & y^3 & x^2 & 0 \\
            & & & & y^2 & -xy^2 & xy & -x \\
            & & & & x^2 & xy^2 & -x^2y^2 - y^4 & 0 \\
            & & & & 0 & x & 0 & -y \\
            \hdottedline
            -xy^2 & -y^3 & -x & xy^2 & \Block{4-4}<\large \boldmath>{z} & & & \\
            -y^2 & xy & 0 & -x^2 & & & & \\
            -x & 0 & y & 0 & & & & \\
            -xy & x^2 & 0 & x^2y^2 + y^4 & & & &
        \end{pNiceArray} \mright.$,} \\
            \tag*{\footnotesize $
                \mleft. \begin{pNiceArray}{cccc:cccc}[margin]
                    \Block{4-4}<\large \boldmath>{z} & & & & -xy & -y^3 & -x^2 & 0 \\
                    & & & & -y^2 & xy^2 & -xy & x \\
                    & & & & -x^2 & -xy^2 & x^2y^2 + y^4 & 0 \\
                    & & & & 0 & -x & 0 & y \\
                    \hdottedline
                    xy^2 & y^3 & x & -xy^2 & \Block{4-4}<\large \boldmath>{z} & & & \\
                    y^2 & -xy & 0 & x^2 & & & & \\
                    x & 0 & -y & 0 & & & & \\
                    xy & -x^2 & 0 & -x^2y^2 - y^4 & & & &
                \end{pNiceArray} \mright\rparen
            $,} \\
    &\M_{6} \coloneqq \mbox{\footnotesize $\expar{
        \begin{pNiceArray}{ccc:ccc}[margin]
            \Block{3-3}<\normalsize \boldmath>{z} & & & xy & x^2 & -y^4 \\
            & & & x^2 & -x^2y^2 - y^4 & -xy^3 \\
            & & & y^2 & xy & x^2 + xy^3 \\
            \hdottedline
            -xy^2 & -x & -y^3 & \Block{3-3}<\normalsize \boldmath>{z} & & \\
            -x & y & 0 & & & \\
            y & 0 & -x & & &
        \end{pNiceArray},
        \begin{pNiceArray}{ccc:ccc}[margin]
            \Block{3-3}<\normalsize \boldmath>{z} & & & -xy & -x^2 & y^4 \\
            & & & -x^2 & x^2y^2 + y^4 & xy^3 \\
            & & & -y^2 & -xy & -x^2 - xy^3 \\
            \hdottedline
            xy^2 & x & y^3 & \Block{3-3}<\normalsize \boldmath>{z} & & \\
            x & -y & 0 & & & \\
            -y & 0 & x & & &
        \end{pNiceArray}
    }$}, \\
    &\M_{7} \coloneqq \mbox{\footnotesize $\expar{
        \begin{pNiceArray}{cc:cc}[margin]
            \Block{2-2}<\small \boldmath>{z} & & -x^2 & y \\
            & & -x^2y^2 - y^4 & -x \\
            \hdottedline
            x & y & \Block{2-2}<\small \boldmath>{z} & \\
            -x^2y^2 - y^4 & x^2 & &
        \end{pNiceArray},
        \begin{pNiceArray}{cc:cc}[margin]
            \Block{2-2}<\small \boldmath>{z} & & x^2 & -y \\
            & & x^2y^2 + y^4 & x \\
            \hdottedline
            -x & -y & \Block{2-2}<\small \boldmath>{z} & \\
            x^2y^2 + y^4 & -x^2 & &
        \end{pNiceArray}
    }$}, \\
    &\M_{8} \coloneqq \mbox{\footnotesize $\expar{
        \begin{pNiceArray}{ccc:ccc}[margin]
            \Block{3-3}<\normalsize \boldmath>{z} & & & -x^2 & xy^2 & -y^4 \\
            & & & -x^2y - y^3 & -x^2 & xy^2 \\
            & & & xy & -y^3 & -x^2 - xy^3 \\
            \hdottedline
            x & y^2 & 0 & \Block{3-3}<\normalsize \boldmath>{z} & & \\
            -xy & x & y^2 & & & \\
            y & 0 & x & & &
        \end{pNiceArray},
        \begin{pNiceArray}{ccc:ccc}[margin]
            \Block{3-3}<\normalsize \boldmath>{z} & & & x^2 & -xy^2 & y^4 \\
            & & & x^2y + y^3 & x^2 & -xy^2 \\
            & & & -xy & y^3 & x^2 + xy^3 \\
            \hdottedline
            -x & -y^2 & 0 & \Block{3-3}<\normalsize \boldmath>{z} & & \\
            xy & -x & -y^2 & & & \\
            -y & 0 & -x & & &
        \end{pNiceArray}
    }$}. \\
\end{align*}

\subsection{\texorpdfstring{$E_{8}^2$ in characteristic $3$: $f = z^2 + x^3 + y^5 + x^2y^2$}{E82InChar3}}
\label{subsec:E_8^2InChar3}
\begin{align*}
    &\M_{1} \coloneqq \mbox{\footnotesize $\expar{
        \begin{pNiceArray}{cc:cc}[margin]
            \Block{2-2}<\small \boldmath>{z} & & -x^2 - xy^2 & y^2 \\
            & & -y^3 & -x \\
            \hdottedline
            x & y^2 & \Block{2-2}<\small \boldmath>{z} & \\
            -y^3 & x^2 + xy^2 & &
        \end{pNiceArray},
        \begin{pNiceArray}{cc:cc}[margin]
            \Block{2-2}<\small \boldmath>{z} & & x^2 + xy^2 & -y^2 \\
            & & y^3 & x \\
            \hdottedline
            -x & -y^2 & \Block{2-2}<\small \boldmath>{z} & \\
            y^3 & -x^2 - xy^2 & &
        \end{pNiceArray}
    }$}, \\
    &\M_{2} \coloneqq \mbox{\footnotesize $\mleft\lparen
        \begin{pNiceArray}{cccc:cccc}[margin] 
            \Block{4-4}<\large \boldmath>{z} & & & & 0 & -y^3 & - x^2 - xy^2 & -y^4 \\
            & & & & y^2 & xy & xy & -x^2 \\
            & & & & -x & y^2 & 0 & 0 \\
            & & & & 0 & x & -y^2 & xy \\
            \hdottedline
            0 & -y^3 & x^2 + xy^2 & -xy^2 & \Block{4-4}<\large \boldmath>{z} & & & \\
            0 & -xy & -y^3 & -x^2 & & & & \\
            x & 0 & 0 & y^3 & & & & \\
            y & x & y^2 & -xy & & & &
        \end{pNiceArray} \mright.$,} \\
            \tag*{\footnotesize $
                \mleft. \begin{pNiceArray}{cccc:cccc}[margin]
                    \Block{4-4}<\large \boldmath>{z} & & & & 0 & y^3 & x^2 + xy^2 & y^4 \\
                    & & & & -y^2 & -xy & -xy & x^2 \\
                    & & & & x & -y^2 & 0 & 0 \\
                    & & & & 0 & -x & y^2 & -xy \\
                    \hdottedline
                    0 & y^3 & -x^2 - xy^2 & xy^2 & \Block{4-4}<\large \boldmath>{z} & & & \\
                    0 & xy & y^3 & x^2 & & & & \\
                    -x & 0 & 0 & -y^3 & & & & \\
                    -y & -x & -y^2 & xy & & & &
                \end{pNiceArray} \mright\rparen
            $,} \\
    &\M_{3} \coloneqq \mbox{\footnotesize $\mleft\lparen
            \begin{pNiceArray}{cccccc:cccccc}[margin]
                \Block{6-6}<\LARGE \boldmath>{z} & & & & & & 0 & x^2 + y^3 & 0 & -x^2 & 0 & 0 \\
                & & & & & & 0 & 0 & y^3 & x^2 + y^3 & -x^2 - xy^2 & -xy^2 \\
                & & & & & & -y^2 & -xy & 0 & -xy & 0 & x^2 \\
                & & & & & & -x & -y^2 & 0 & -xy & 0 & -x^2 - y^3 \\
                & & & & & & 0 & -x & 0 & -y^2 & 0 & 0 \\
                & & & & & & y & y^2 & -x & -xy & -y^2 & 0 \\
                \hdottedline
                xy & 0 & x^2 + y^3 & x^2 & -xy^2 & 0 & \Block{6-6}<\LARGE \boldmath>{z} & & & & & \\
                -y^2 & 0 & 0 & 0 & x^2 & 0 & & & & & & \\
                -xy & -y^2 & xy & xy & -x^2y - y^3 & x^2 + xy^2 & & & & & & \\
                x & 0 & 0 & 0 & x^2 + y^3 & 0 & & & & & & \\
                x - y^2 & x & y^2 & 0 & x^2 - xy^2 & y^3 & & & & & & \\
                y & 0 & -x & y^2 & -xy & 0 & & & & & &
            \end{pNiceArray} \mright.$,} \\
            \tag*{\footnotesize $
                \mleft. \begin{pNiceArray}{cccccc:cccccc}[margin]
                    \Block{6-6}<\LARGE \boldmath>{z} & & & & & & 0 & -x^2 - y^3 & 0 & x^2 & 0 & 0 \\
                    & & & & & & 0 & 0 & -y^3 & -x^2 - y^3 & x^2 + xy^2 & xy^2 \\
                    & & & & & & y^2 & xy & 0 & xy & 0 & -x^2 \\
                    & & & & & & x & y^2 & 0 & xy & 0 & x^2 + y^3 \\
                    & & & & & & 0 & x & 0 & y^2 & 0 & 0 \\
                    & & & & & & -y & -y^2 & x & xy & y^2 & 0 \\
                    \hdottedline
                    -xy & 0 & -x^2 - y^3 & -x^2 & xy^2 & 0 & \Block{6-6}<\LARGE \boldmath>{z} & & & & & \\
                    y^2 & 0 & 0 & 0 & -x^2 & 0 & & & & & & \\
                    xy & y^2 & -xy & -xy & x^2y + y^3 & -x^2 - xy^2 & & & & & & \\
                    -x & 0 & 0 & 0 & -x^2 - y^3 & 0 & & & & & & \\
                    -x + y^2 & -x & -y^2 & 0 & -x^2 + xy^2 & -y^3 & & & & & & \\
                    -y & 0 & x & -y^2 & xy & 0 & & & & & &
                \end{pNiceArray} \mright\rparen
            $,}\\
    &\M_{4} \coloneqq \mleft\lparen
        \begin{pNiceArray}{ccccc:ccccc}[margin] 
            \Block{5-5}<\Large \boldmath>{z} & & & & & 0 & -x^2 - xy^2 & -y^3 & 0 & xy^2 \\
            & & & & & y^3 & 0 & -x^2 & -x^2y & -x^2y \\
            & & & & & -y^2 & xy & 0 & x^2 & -y^3 \\
            & & & & & 0 & y^3 & -xy & 0 & x^2 \\
            & & & & & x + y^2 & 0 & y^2 & y^3 & y^3 \\
            \hdottedline
            0 & -y^2 & 0 & 0 & -x^2 & \Block{5-5}<\Large \boldmath>{z} & & & & \\
            x & 0 & 0 & -y^2 & 0 & & & & & \\
            y^2 & x & xy & xy & 0 & & & & & \\
            -y & 0 & -x - y^2 & 0 & -y^2 & & & & & \\
            0 & y & y^2 & -x & 0 & & & & &
        \end{pNiceArray} \mright., \\
            \tag*{$
                \mleft. \begin{pNiceArray}{ccccc:ccccc}[margin]
                    \Block{5-5}<\Large \boldmath>{z} & & & & & 0 & x^2 + xy^2 & y^3 & 0 & -xy^2 \\
                    & & & & & -y^3 & 0 & x^2 & x^2y & x^2y \\
                    & & & & & y^2 & -xy & 0 & -x^2 & y^3 \\
                    & & & & & 0 & -y^3 & xy & 0 & -x^2 \\
                    & & & & & -x - y^2 & 0 & -y^2 & -y^3 & -y^3 \\
                    \hdottedline
                    0 & y^2 & 0 & 0 & x^2 & \Block{5-5}<\Large \boldmath>{z} & & & & \\
                    -x & 0 & 0 & y^2 & 0 & & & & & \\
                    -y^2 & -x & -xy & -xy & 0 & & & & & \\
                    y & 0 & x + y^2 & 0 & y^2 & & & & & \\
                    0 & -y & -y^2 & x & 0 & & & & &
                \end{pNiceArray} \mright\rparen
            $,} \\
    &\M_{5} \coloneqq \mleft\lparen
        \begin{pNiceArray}{cccc:cccc}[margin] 
            \Block{4-4}<\large \boldmath>{z} & & & & 0 & x & 0 & -y \\
            & & & & y^2 & xy & xy & x \\
            & & & & x^2 & 0 & -x^2y - y^4 & -y^3 \\
            & & & & xy & -y^3 & x^2 & 0 \\
            \hdottedline
            0 & -y^3 & -x & -xy & \Block{4-4}<\large \boldmath>{z} & & & \\
            -x^2 & -xy & 0 & y^2 & & & & \\
            -y^3 & 0 & y & -x & & & & \\
            x^2y + y^4 & -x^2 & 0 & xy & & & &
        \end{pNiceArray} \mright., \\
            \tag*{$
                \mleft. \begin{pNiceArray}{cccc:cccc}[margin]
                    \Block{4-4}<\large \boldmath>{z} & & & & 0 & -x & 0 & y \\
                    & & & & -y^2 & -xy & -xy & -x \\
                    & & & & -x^2 & 0 & x^2y + y^4 & y^3 \\
                    & & & & -xy & y^3 & -x^2 & 0 \\
                    \hdottedline
                    0 & y^3 & x & xy & \Block{4-4}<\large \boldmath>{z} & & & \\
                    x^2 & xy & 0 & -y^2 & & & & \\
                    y^3 & 0 & -y & x & & & & \\
                    -x^2y - y^4 & x^2 & 0 & -xy & & & &
                \end{pNiceArray} \mright\rparen
            $,} \\
    &\M_{6} \coloneqq \mbox{\footnotesize $\expar{
        \begin{pNiceArray}{ccc:ccc}[margin]
            \Block{3-3}<\normalsize \boldmath>{z} & & & x^2 & -x^2y - y^4 & -xy^3 \\
            & & & xy & x^2 & -y^4 \\
            & & & y^2 & xy & x^2 + xy^2 \\
            \hdottedline
            -x & -xy & -y^3 & \Block{3-3}<\normalsize \boldmath>{z} & & \\
            y & -x & 0 & & & \\
            0 & y & -x & & &
        \end{pNiceArray},
        \begin{pNiceArray}{ccc:ccc}[margin]
            \Block{3-3}<\normalsize \boldmath>{z} & & & -x^2 & x^2y + y^4 & xy^3 \\
            & & & -xy & -x^2 & y^4 \\
            & & & -y^2 & -xy & -x^2 - xy^2 \\
            \hdottedline
            x & xy & y^3 & \Block{3-3}<\normalsize \boldmath>{z} & & \\
            -y & x & 0 & & & \\
            0 & -y & x & & &
        \end{pNiceArray}
    }$}, \\
    &\M_{7} \coloneqq \mbox{\footnotesize $\expar{
        \begin{pNiceArray}{cc:cc}[margin]
            \Block{2-2}<\small \boldmath>{z} & & -x^2 & y \\
            & & -x^2y - y^4 & -x \\
            \hdottedline
            x & y & \Block{2-2}<\small \boldmath>{z} & \\
            -x^2y - y^4 & x^2 & &
        \end{pNiceArray},
        \begin{pNiceArray}{cc:cc}[margin]
            \Block{2-2}<\small \boldmath>{z} & & x^2 & -y \\
            & & x^2y + y^4 & x \\
            \hdottedline
            -x & -y & \Block{2-2}<\small \boldmath>{z} & \\
            x^2y + y^4 & -x^2 & &
        \end{pNiceArray}
    }$}, \\
    &\M_{8} \coloneqq \mbox{\footnotesize $\expar{
        \begin{pNiceArray}{ccc:ccc}[margin]
            \Block{3-3}<\normalsize \boldmath>{z} & & & -x^2 & xy^2 & -y^4 \\
            & & & -x^2 - y^3 & -x^2 & xy^2 \\
            & & & xy & -y^3 & -x^2 - xy^2 \\
            \hdottedline
            x & y^2 & 0 & \Block{3-3}<\normalsize \boldmath>{z} & & \\
            -x & x & y^2 & & & \\
            y & 0 & x & & &
        \end{pNiceArray},
        \begin{pNiceArray}{ccc:ccc}[margin]
            \Block{3-3}<\normalsize \boldmath>{z} & & & x^2 & -xy^2 & y^4 \\
            & & & x^2 + y^3 & x^2 & -xy^2 \\
            & & & -xy & y^3 & x^2 + xy^2 \\
            \hdottedline
            -x & -y^2 & 0 & \Block{3-3}<\normalsize \boldmath>{z} & & \\
            x & -x & -y^2 & & & \\
            -y & 0 & -x & & &
        \end{pNiceArray}
    }$}. \\
\end{align*}

\subsection{\texorpdfstring{$E_{8}^r$ in characteristic $2$: $f = z^2 + x^3 + y^5 + zg$}{E8rInChar2}}
\label{subsec:E_8^rInChar2}
Here, 
\[ 
    g \coloneqq
    \begin{cases}
        0 & \text{if $r = 0$,} \\
        xy^3 & \text{if $r = 1$,} \\
        xy^2 & \text{if $r = 2$,} \\
        y^3 & \text{if $r = 3$,} \\
        xy & \text{if $r = 4$.}
    \end{cases}
\]

\begin{align*}
    &\M_{1} \coloneqq \mbox{\footnotesize $\expar{
        \begin{pNiceArray}{cc:w{c}{2em}w{c}{1em}}[margin]
            \Block{2-2}<\small \boldmath>{z} & & y^3 & x^2 \\
            & & x & y^2 \\
            \hdottedline
            y^2 & x^2 & \Block{2-2}<\small \boldmath>{z + g} & \\
            x & y^3 & &
        \end{pNiceArray},
        \begin{pNiceArray}{w{c}{1em}w{c}{2em}:cc}[margin]
            \Block{2-2}<\small \boldmath>{z + g} & & y^3 & x^2 \\
            & & x & y^2 \\
            \hdottedline
            y^2 & x^2 & \Block{2-2}<\small \boldmath>{z} & \\
            x & y^3 & &
        \end{pNiceArray}
    }$}, \\
    &\M_{2} \coloneqq \mbox{\footnotesize $\expar{
        \begin{pNiceArray}{cccc:cccc}[margin] 
            \Block{4-4}<\large \boldmath>{z} & & & & 0 & y^3 & x^2 & 0 \\
            & & & & y^2 & 0 & xy & x^2 \\
            & & & & x & y^2 & 0 & y^3 \\
            & & & & 0 & x & y^2 & 0 \\
            \hdottedline
            0 & y^3 & x^2 & xy^2 & \Block{4-4}<\large \boldmath>{z + g} & & & \\
            y^2 & 0 & 0 & x^2 & & & & \\
            x & 0 & 0 & y^3 & & & & \\
            y & x & y^2 & 0 & & & &
        \end{pNiceArray},
        \begin{pNiceArray}{cccc:cccc}[margin]
            \Block{4-4}<\large \boldmath>{z + g} & & & & 0 & y^3 & x^2 & 0 \\
            & & & & y^2 & 0 & xy & x^2 \\
            & & & & x & y^2 & 0 & y^3 \\
            & & & & 0 & x & y^2 & 0 \\
            \hdottedline
            0 & y^3 & x^2 & xy^2 & \Block{4-4}<\large \boldmath>{z} & & & \\
            y^2 & 0 & 0 & x^2 & & & & \\
            x & 0 & 0 & y^3 & & & & \\
            y & x & y^2 & 0 & & & &
        \end{pNiceArray}
    }$}, \\
    &\M_{3} \coloneqq \mbox{\footnotesize $\mleft\lparen
        \begin{pNiceArray}{cccccc:cccccc}[margin] 
            \Block{6-6}<\LARGE \boldmath>{z} & & & & & & 0 & 0 & 0 & x^2 & xy^2 & y^4 \\
            & & & & & & 0 & 0 & 0 & y^3 & x^2 & xy^2 \\
            & & & & & & 0 & 0 & 0 & xy & y^3 & x^2 \\
            & & & & & & x & y^2 & 0 & 0 & 0 & y^3 \\
            & & & & & & 0 & x & y^2 & y^2 & 0 & 0 \\
            & & & & & & y & 0 & x & 0 & y^2 & 0 \\
            \hdottedline
            0 & 0 & y^3 & x^2 & xy^2 & y^4 & \Block{6-6}<\LARGE \boldmath>{z + g} & & & & & \\
            y^2 & 0 & 0 & y^3 & x^2 & xy^2 & & & & & & \\
            0 & y^2 & 0 & xy & y^3 & x^2 & & & & & & \\
            x & y^2 & 0 & 0 & 0 & 0 & & & & & & \\
            0 & x & y^2 & 0 & 0 & 0 & & & & & & \\
            y & 0 & x & 0 & 0 & 0 & & & & & &
        \end{pNiceArray} \mright.$,} \\
            \tag*{\footnotesize $
                \mleft. \begin{pNiceArray}{cccccc:cccccc}[margin]
                    \Block{6-6}<\LARGE \boldmath>{z + g} & & & & & & 0 & 0 & 0 & x^2 & xy^2 & y^4 \\
                    & & & & & & 0 & 0 & 0 & y^3 & x^2 & xy^2 \\
                    & & & & & & 0 & 0 & 0 & xy & y^3 & x^2 \\
                    & & & & & & x & y^2 & 0 & 0 & 0 & y^3 \\
                    & & & & & & 0 & x & y^2 & y^2 & 0 & 0 \\
                    & & & & & & y & 0 & x & 0 & y^2 & 0 \\
                    \hdottedline
                    0 & 0 & y^3 & x^2 & xy^2 & y^4 & \Block{6-6}<\LARGE \boldmath>{z} & & & & & \\
                    y^2 & 0 & 0 & y^3 & x^2 & xy^2 & & & & & & \\
                    0 & y^2 & 0 & xy & y^3 & x^2 & & & & & & \\
                    x & y^2 & 0 & 0 & 0 & 0 & & & & & & \\
                    0 & x & y^2 & 0 & 0 & 0 & & & & & & \\
                    y & 0 & x & 0 & 0 & 0 & & & & & &
                \end{pNiceArray} \mright\rparen
            $,} \\
    &\M_{4} \coloneqq \mbox{\footnotesize $\expar{
        \begin{pNiceArray}{ccccc:ccccc}[margin] 
            \Block{5-5}<\Large \boldmath>{z} & & & & & y^3 & x^2 & 0 & 0 & 0 \\
            & & & & & 0 & y^3 & x^2 & xy^2 & y^4 \\
            & & & & & 0 & xy & y^3 & x^2 & xy^2 \\
            & & & & & y^2 & 0 & xy & y^3 & x^2 \\
            & & & & & x & y^2 & 0 & 0 & 0 \\
            \hdottedline
            y^2 & 0 & 0 & 0 & x^2 & \Block{5-5}<\Large \boldmath>{z + g} & & & & \\
            x & 0 & 0 & 0 & y^3 & & & & & \\
            0 & x & y^2 & 0 & 0 & & & & & \\
            y & 0 & x & y^2 & 0 & & & & & \\
            0 & y & 0 & x & y^2 & & & & &
        \end{pNiceArray},
        \begin{pNiceArray}{ccccc:ccccc}[margin]
            \Block{5-5}<\Large \boldmath>{z + g} & & & & & y^3 & x^2 & 0 & 0 & 0 \\
            & & & & & 0 & y^3 & x^2 & xy^2 & y^4 \\
            & & & & & 0 & xy & y^3 & x^2 & xy^2 \\
            & & & & & y^2 & 0 & xy & y^3 & x^2 \\
            & & & & & x & y^2 & 0 & 0 & 0 \\
            \hdottedline
            y^2 & 0 & 0 & 0 & x^2 & \Block{5-5}<\Large \boldmath>{z} & & & & \\
            x & 0 & 0 & 0 & y^3 & & & & & \\
            0 & x & y^2 & 0 & 0 & & & & & \\
            y & 0 & x & y^2 & 0 & & & & & \\
            0 & y & 0 & x & y^2 & & & & &
        \end{pNiceArray}
    }$}, \\
    &\M_{5} \coloneqq \mbox{\footnotesize $\expar{
        \begin{pNiceArray}{cccc:cccc}[margin] 
            \Block{4-4}<\large \boldmath>{z} & & & & xy & y^2 & x^2 & 0 \\
            & & & & y^3 & 0 & 0 & x \\
            & & & & x^2 & 0 & 0 & y^2 \\
            & & & & 0 & x & y^3 & y \\
            \hdottedline
            0 & y^2 & x & 0 & \Block{4-4}<\large \boldmath>{z + g} & & & \\
            y^3 & xy & 0 & x^2 & & & & \\
            x & 0 & y & y^2 & & & & \\
            0 & x^2 & y^3 & 0 & & & &
        \end{pNiceArray},
        \begin{pNiceArray}{cccc:cccc}[margin]
            \Block{4-4}<\large \boldmath>{z + g} & & & & xy & y^2 & x^2 & 0 \\
            & & & & y^3 & 0 & 0 & x \\
            & & & & x^2 & 0 & 0 & y^2 \\
            & & & & 0 & x & y^3 & y \\
            \hdottedline
            0 & y^2 & x & 0 & \Block{4-4}<\large \boldmath>{z} & & & \\
            y^3 & xy & 0 & x^2 & & & & \\
            x & 0 & y & y^2 & & & & \\
            0 & x^2 & y^3 & 0 & & & &
        \end{pNiceArray}
    }$}, \\
    &\M_{6} \coloneqq \mbox{\footnotesize $\expar{
        \begin{pNiceArray}{ccc:ccc}[margin]
            \Block{3-3}<\normalsize \boldmath>{z} & & & x^2 & y^4 & xy^3 \\
            & & & xy & x^2 & y^4 \\
            & & & y^2 & xy & x^2 \\
            \hdottedline
            x & 0 & y^3 & \Block{3-3}<\normalsize \boldmath>{z + g} & & \\
            y & x & 0 & & & \\
            0 & y & x & & &
        \end{pNiceArray},
        \begin{pNiceArray}{ccc:ccc}[margin]
            \Block{3-3}<\normalsize \boldmath>{z + g} & & & x^2 & y^4 & xy^3 \\
            & & & xy & x^2 & y^4 \\
            & & & y^2 & xy & x^2 \\
            \hdottedline
            x & 0 & y^3 & \Block{3-3}<\normalsize \boldmath>{z} & & \\
            y & x & 0 & & & \\
            0 & y & x & & &
        \end{pNiceArray}
    }$}, \\
    &\M_{7} \coloneqq \mbox{\footnotesize $\expar{
        \begin{pNiceArray}{cc:w{c}{2em}w{c}{1em}}[margin]
            \Block{2-2}<\small \boldmath>{z} & & x^2 & y^4 \\
            & & y & x \\
            \hdottedline
            x & y^4 & \Block{2-2}<\small \boldmath>{z + g} & \\
            y & x^2 & &
        \end{pNiceArray},
        \begin{pNiceArray}{w{c}{1em}w{c}{2em}:cc}[margin]
            \Block{2-2}<\small \boldmath>{z + g} & & x^2 & y^4 \\
            & & y & x \\
            \hdottedline
            x & y^4 & \Block{2-2}<\small \boldmath>{z} & \\
            y & x^2 & &
        \end{pNiceArray}
    }$}, \\
    &\M_{8} \coloneqq \mbox{\footnotesize $\expar{
        \begin{pNiceArray}{ccc:ccc}[margin]
            \Block{3-3}<\normalsize \boldmath>{z} & & & x^2 & xy^2 & y^4 \\
            & & & y^3 & x^2 & xy^2 \\
            & & & xy & y^3 & x^2 \\
            \hdottedline
            x & y^2 & 0 & \Block{3-3}<\normalsize \boldmath>{z + g} & & \\
            0 & x & y^2 & & & \\
            y & 0 & x & & &
        \end{pNiceArray},
        \begin{pNiceArray}{ccc:ccc}[margin]
            \Block{3-3}<\normalsize \boldmath>{z + g} & & & x^2 & xy^2 & y^4 \\
            & & & y^3 & x^2 & xy^2 \\
            & & & xy & y^3 & x^2 \\
            \hdottedline
            x & y^2 & 0 & \Block{3-3}<\normalsize \boldmath>{z} & & \\
            0 & x & y^2 & & & \\
            y & 0 & x & & &
        \end{pNiceArray}
    }$}. 
\end{align*}
\endgroup

\bibliographystyle{alphaurl}
\bibliography{References}
\end{document}